\documentclass[11pt,a4paper]{article}
\usepackage{latexsym,amssymb,amsfonts,amsmath,amsthm,nccmath,float,graphicx,multirow,enumitem,setspace,bm,authblk,tabularx,longtable,scrextend,anyfontsize}
\usepackage[dvipsnames]{xcolor}
\usepackage[hidelinks]{hyperref}
\usepackage[margin=2cm]{geometry}

\changefontsizes[13.9pt]{11.3pt}

\hypersetup{
	colorlinks,
	linkcolor={red!80!black},
	citecolor={blue!80!black},
	urlcolor={blue!80!black}
}

\allowdisplaybreaks
\setlist[itemize]{topsep=4pt,itemsep=2pt}
\setlist[enumerate]{topsep=4pt,itemsep=2pt}

\newtheorem{theorem}{Theorem}[section]
\newtheorem{lemma}[theorem]{Lemma}
\newtheorem{proposition}[theorem]{Proposition}
\newtheorem{corollary}[theorem]{Corollary}
\theoremstyle{definition}
\newtheorem{definition}[theorem]{Definition}
\theoremstyle{remark}
\newtheorem{remark}[theorem]{Remark}
\theoremstyle{definition}
\newtheorem{condition}[theorem]{Condition}
\newtheorem{example}[theorem]{Example}

\newcommand{\ind}{\mathbf 1}

\allowdisplaybreaks

\let\oldbibliography\thebibliography
\renewcommand{\thebibliography}[1]{
  \oldbibliography{#1}
  \setlength{\itemsep}{1pt}
  \setlength{\parsep}{0pt}
  \small
}

\title{The existence of symmetric cyclic Steiner quadruple systems}
\author[a]{Hengrui Liu}
\author[a,b]{Tao Feng \thanks{Supported by NSFC under Grant 12271023}}
\affil[a]{School of Mathematics and Statistics, Beijing Jiaotong University, Beijing, 100044, P.R. China}
\affil[b]{Hebei Provincial Key Laboratory of Mathematical Theory and Analysis for Network and Data Science, Beijing Jiaotong University, Beijing, 100044, PR China}

\affil[ ]{henryleo@bjtu.edu.cn; tfeng@bjtu.edu.cn}
\date{}

\begin{document}

\maketitle

\begin{abstract}
A Steiner quadruple system of order $v$, denoted by $\operatorname{SQS}(v)$, is called cyclic if it admits $\mathbb Z_v$ as a group of automorphisms acting regularly on its point set. A cyclic system is called symmetric cyclic, or S-cyclic, if the multiplier $-1$ preserves each block orbit individually. Grannell and Griggs conjectured in 1983 that such a system of order $v$ exists if and only if $v\equiv2,4,10,20\pmod{24}$ and every odd prime divisor $p$ of $v$ satisfies $p\equiv1$ or $5\pmod{12}$. This paper proves the conjecture. The main step is to construct an $S$-cyclic Steiner quadruple system of order $2p$ for every prime $p\equiv1\pmod4$. Together with existing reduction theorems, these constructions determine which finite abelian groups admit a symmetric group-invariant Steiner quadruple system.
\end{abstract}




\section{Introduction}

A \emph{Steiner quadruple system} of order $v$, denoted by $\operatorname{SQS}(v)$, is a pair $(V,\mathcal B)$, where $V$ is a set of $v$ points and $\mathcal B$ is a collection of $4$-subsets of $V$, called blocks, such that every $3$-subset of $V$ is contained in exactly one block. A necessary condition for the existence of an $\operatorname{SQS}(v)$ is $v\equiv2$ or $4\pmod6$. The smallest nontrivial example, $\operatorname{SQS}(8)$, was already known to Kirkman in 1847 \cite{Kirkman1847}, while the existence problem was completely settled by Hanani \cite{Hanani} in 1960, who proved that
the necessary condition is also sufficient.

An automorphism of an $\operatorname{SQS}(v)$ is a permutation of its point set that maps blocks to blocks. The set of all automorphisms of an $\operatorname{SQS}(v)$ forms a group under composition, called the \emph{full automorphism group}. A subgroup of the full automorphism group is called a \emph{group of automorphisms}. An $\operatorname{SQS}(v)$ is called \emph{cyclic}, written as $\operatorname{CSQS}(v)$, if its point set can be identified with the cyclic group $\mathbb Z_v$ of order $v$ such that the translation $x\mapsto x+1$ is an automorphism. In this case, $\mathbb Z_v$ acts regularly on the point set and induces an action on the block set. The blocks are partitioned into orbits under this action. A representative chosen from each orbit is called a \emph{base block}. A block orbit is called \emph{full} if its length is $v$, and \emph{short} otherwise. A cyclic Steiner quadruple system is called \emph{strictly cyclic} if all its block orbits are full; such a system is denoted by $\operatorname{sSQS}(v)$.

The existence problems for cyclic and strictly cyclic Steiner quadruple systems remain far from settled. A cyclic $\operatorname{SQS}(v)$ can exist only when $v\equiv2$ or $4\pmod6$, while a strictly cyclic $\operatorname{SQS}(v)$ satisfies the stronger necessary condition $v\equiv2$ or $10\pmod{24}$. Even within a fixed admissible congruence class of $v$, the existence problems remain unresolved. See \cite{FengChang2011} and \cite{FengChangJi2008} for known results on cyclic and strictly cyclic Steiner quadruple systems, respectively.

Lindner and Rosa~\cite{LindnerRosa1978} introduced the term \emph{symmetric cyclic} for cyclic Steiner quadruple systems satisfying the following symmetry condition. A cyclic $\operatorname{SQS}(v)$ is called \emph{symmetric cyclic}, or \emph{$S$-cyclic} for short, if the map $x\mapsto -x$ leaves each $\mathbb Z_v$-orbit of blocks invariant. Equivalently, for every block $B$, there exists $\delta\in\mathbb Z_v$ such that $-B=B+\delta$. Grannell and Griggs~\cite{GrannellGriggs}, and independently Diener~\cite{Diener}, established the following necessary conditions.

\begin{lemma}[\cite{Diener,GrannellGriggs}]\label{lem:intro-necessary}
If an $S$-cyclic $\operatorname{SQS}(v)$ exists, then $v\equiv2,4,10,20\pmod{24}$, and every odd prime divisor $p$ of $v$ satisfies $p\equiv1$ or $5\pmod{12}$.
\end{lemma}

Grannell and Griggs \cite[Section~3]{GrannellGriggsSurvey} conjectured that the necessary conditions in Lemma \ref{lem:intro-necessary} are sufficient. The same question was posed by Diener \cite{Diener} and was later included among the open problems in the survey of Hartman and Phelps \cite[Section~12.3, Problem 6]{HartmanPhelps1992}. Piotrowski \cite{Piotrowski} established the following reduction, which we state in an equivalent form for our purposes.

\begin{lemma}[{\cite[Satz~14.1]{Piotrowski}}]\label{lem:intro-piotrowski-equiv}
An $S$-cyclic $\operatorname{SQS}(v)$ exists if and only if $v\equiv2,4,10,20\pmod{24}$, every odd prime divisor $p$ of $v$ satisfies $p\equiv1$ or $5\pmod{12}$, and an $S$-cyclic $\operatorname{SQS}(2p)$ exists for every such prime $p$.
\end{lemma}

It follows from Lemma \ref{lem:intro-piotrowski-equiv} that the existence problem for $S$-cyclic Steiner quadruple systems is reduced to the construction of $S$-cyclic $\operatorname{SQS}(2p)$ for primes $p\equiv1,5\pmod{12}$. The construction of these systems of order $2p$ has a long history. Fitting~\cite{Fitting1915} introduced a graph-theoretic approach to the construction of cyclic Steiner quadruple systems in 1915, which was subsequently developed by K\"ohler~\cite{Kohler} in 1979. K\"ohler introduced the graph $\mathcal K_v$ for the construction of $S$-cyclic $\operatorname{SQS}(v)$. For $v\equiv2,10\pmod{24}$, a $1$-factor in $\mathcal K_v$ yields a strictly cyclic and $S$-cyclic $\operatorname{SQS}(v)$. For $v=2p>10$, where $p$ is a prime with $p\equiv1$ or $5\pmod{12}$, $\mathcal K_{2p}$ has two components, one of which has an explicit $1$-factor. Thus the problem reduces to finding a $1$-factor in the remaining component, which is isomorphic to $\mathcal K_p$ as described in Section~\ref{sec:auxiliary}. When $p\equiv5\pmod{12}$, K\"ohler further used affine orbits to obtain a smaller graph, and constructed $S$-cyclic $\operatorname{SQS}(2p)$ for all primes $p\equiv5\pmod{12}$ with $p\le233$; see \cite[Section~3]{GrannellGriggsSurvey} for more details.

Siemon~\cite{Siemon1987,Siemon1989,SiemonKohler} developed K\"ohler's orbit-graph approach, studying the structure of the resulting graphs and their $1$-factors. In particular, he related
the existence of the required $1$-factors to the bridgelessness of certain K\"ohler graphs and their reduced graphs. For primes $p\equiv53,77\pmod{120}$, Siemon~\cite{Siemon1991} reduced the construction of $S$-cyclic $\operatorname{SQS}(2p)$ to a number-theoretic condition and verified this condition for all admissible primes $p<500000$. His later work~\cite{Siemon1998} further related the bridgelessness
of the relevant orbit graphs to a number-theoretic conjecture on complete intervals, which remained unresolved in full generality.

Bitan and Etzion~\cite{BitanEtzion1993} applied the K\"ohler orbit-graph method to orders $4p$ and used it to construct $S$-cyclic $\operatorname{SQS}(4p)$ for every prime $p\equiv5\pmod{12}$ with $p<1500000$. 

Despite these developments, the existence of an $S$-cyclic $\operatorname{SQS}(2p)$ for every prime $p\equiv1$ or $5\pmod{12}$ remained open. In this paper, we extend K\"ohler's orbit-graph method to establish the existence of an $S$-cyclic $\operatorname{SQS}(2p)$ for every such prime $p$. By Lemma~\ref{lem:intro-piotrowski-equiv}, this settles the existence problem for $S$-cyclic Steiner quadruple systems.

The main results of this paper are stated below.

\begin{theorem}\label{thm:spectrum-main}
An $S$-cyclic $\operatorname{SQS}(v)$ exists if and only if $v\equiv2,4,10,20\pmod{24}$ and every odd prime divisor $p$ of $v$ satisfies $p\equiv1$ or $5\pmod{12}$.
\end{theorem}

Every $S$-cyclic $\operatorname{SQS}(v)$ with $v\equiv2,10\pmod{24}$ is strictly cyclic (cf. \cite[Theorem~2]{ChangEtAl2017}). We call an $\operatorname{SQS}(v)$ that is both $S$-cyclic and strictly cyclic \emph{strictly $S$-cyclic}. Theorem~\ref{thm:spectrum-main} therefore yields the following existence result.

\begin{theorem}\label{thm:strict-spectrum-main}
A strictly $S$-cyclic $\operatorname{SQS}(v)$ exists if and only if $v\equiv2,10\pmod{24}$ and every odd prime divisor $p$ of $v$ satisfies $p\equiv1$ or $5\pmod{12}$.
\end{theorem}

The definition of an $S$-cyclic Steiner quadruple system extends naturally from cyclic groups to finite abelian groups. Let $(K,+)$ be a finite abelian group. A Steiner quadruple system $(K,\mathcal B)$ is called \emph{$K$-invariant} if $B+x\in\mathcal B$ for every $B\in\mathcal B$ and $x\in K$. We call a $K$-invariant Steiner quadruple system \emph{symmetric $K$-invariant} if, for every block $B$, there exists $y\in K$ such that $-B=B+y$. Taking $K=\mathbb Z_v$ recovers the definition of an $S$-cyclic $\operatorname{SQS}(v)$.

Ji and Lu~\cite{JiLu} showed that, once the necessary conditions on $K$ are satisfied, the existence of symmetric $K$-invariant Steiner quadruple systems reduces to the construction of $S$-cyclic
$\operatorname{SQS}(2p)$ for each odd prime divisor $p$ of $|K|$. Combining their reduction with Theorem~\ref{thm:spectrum-main} gives the following theorem.

\begin{theorem}\label{thm:abelian-spectrum}
Let $K$ be a finite abelian group of order $v$. A symmetric $K$-invariant $\operatorname{SQS}(v)$ exists if and only if $v\equiv2,4,8,10,16,20\pmod{24}$, no element of $K$ has order divisible by $8$, and every odd prime divisor $p$ of $v$ satisfies $p\equiv1$ or $5\pmod{12}$. Moreover, if a symmetric $K$-invariant $\operatorname{SQS}(v)$ exists, then every block orbit under $K$ has length $v$ if and only if $v\equiv2,10\pmod{24}$.
\end{theorem}

The remainder of this paper is organized as follows. Section~\ref{sec:kohler-method} reviews K\"ohler's construction and its reduction to a perfect matching problem. Section~\ref{sec:auxiliary} introduces the auxiliary covering problem and its affine normalization. Sections~\ref{sec:p5} and~\ref{sec:p1} treat the cases $p\equiv5\pmod{12}$ and $p\equiv1\pmod{12}$, respectively. Section~\ref{sec:existence} completes the constructions and proves the three main theorems. Section~\ref{sec:concluding-remarks} discusses further questions on cyclic Steiner quadruple systems and optical orthogonal codes.

\section{K\"ohler's method revisited}\label{sec:kohler-method}

This section reviews K\"ohler's orbit-graph method for constructing Steiner quadruple systems of order $2p$~\cite{Kohler,Siemon1987,SiemonKohler,Siemon1991}. We describe the prescribed block orbits and the triples they cover. The remaining covering problem is formulated as a $1$-factor problem in Section~\ref{sec:auxiliary}.

Let $p\ge5$ be prime. Following the geometric description in~\cite[VIII.10.2]{BethJungnickelLenz}, place the elements of $\mathbb Z_{2p}$ in their cyclic order at the vertices of a regular $2p$-gon. A $3$-subset (triple) of $\mathbb Z_{2p}$ is then represented by the triangle whose vertices are its three elements. Translation of the labels corresponds to rotation of the polygon, whereas negation corresponds to reflection. Thus, translation orbits of triples in $\mathbb Z_{2p}$ correspond to congruence classes of triangles under rotations.

A triple in $\mathbb Z_{2p}$ containing an antipodal pair, that is, two points differing by $p$ modulo $2p$, is represented by a \emph{right-angled triangle}. A triple of the form $\{j-d,j,j+d\}$, where $j,d\in\mathbb Z_{2p}$ and $d\notin\{0,p\}$, is represented by an \emph{isosceles triangle}. These two classes are disjoint since $2p\equiv2\pmod4$. The remaining triples are called \emph{irregular triangles}; see~\cite[VIII.10.3]{BethJungnickelLenz}.

A $4$-subset of $\mathbb Z_{2p}$ is called a \emph{trapezoid} if it has the form $\{x,x+a,x+b,x+a+b\}$ for some $x,a,b\in\mathbb Z_{2p}$. A $4$-subset of $\mathbb Z_{2p}$ is called a \emph{kite quadrangle} if it has the form $\{x,x+p,x+a,x-a\}$ for some $x,a\in\mathbb Z_{2p}$ with $a\notin\{0,p\}$; geometrically, its sides form two pairs of equal adjacent sides.

K\"ohler's construction prescribes the translation orbits of kite quadrangles and covers the remaining irregular triangles by trapezoids. For order $2p$, the latter problem splits according to parity~\cite[Section~3, p.~413]{GrannellGriggsSurvey}. Accordingly, the construction consists of the following three parts:
\begin{enumerate}[label=\Roman*.,leftmargin=*]
\item Kite quadrangles covering every right-angled and every isosceles triangle exactly once.
\item Trapezoids with two vertices of each parity, covering every irregular triangle containing points of both parities exactly once.
\item Trapezoids whose four vertices have the same parity, covering every irregular triangle whose vertices have the same parity exactly once.
\end{enumerate}
In Types~II and~III, only trapezoids containing no right-angled or isosceles triangle are used. The three classes of triples partition all triples in $\mathbb Z_{2p}$, so families with these covering properties together form an $\operatorname{SQS}(2p)$. For every kite quadrangle or trapezoid $B$, there exists $c\in\mathbb Z_{2p}$ such that $-B+c=B$. Hence negation preserves each selected translation orbit, and the resulting system is $S$-cyclic.

The Type~I and Type~II families have the following direct constructions.

\begin{lemma}[cf.~{\cite[Lemmas~4.2(ii) and~4.10]{MunemasaSawa}}]\label{lem:type-I}
Let $p\ge5$ be prime. The base blocks $\{0,p,i,-i\}$ in $\mathbb Z_{2p}$, where $i$ is an integer with $1\le i\le(p-1)/2$, represent full translation orbits covering every right-angled and every isosceles triangle exactly once and no irregular triangle.
\end{lemma}

\begin{proof}
The translation orbits of right-angled triangles have unique representatives $\{0,p,a\}$ with $1\le a\le p-1$. Since $p\ne3$, every isosceles triangle has a unique apex, so its translation orbit has a unique representative $\{0,a,-a\}$ in the same range. For each $i$, the two right-angled triangles and two isosceles triangles in $\{0,p,i,-i\}$ belong to the translation orbits with parameters $i$ and $p-i$ in the respective lists. These pairs $\{i,p-i\}$ partition $\{1,\ldots,p-1\}$ as $i$ ranges from $1$ to $(p-1)/2$.

In particular, the four triangles in each base block lie in distinct translation orbits. Any translation preserving a base block must therefore fix each of its four triangles setwise; since $3\nmid2p$, every triple in $\mathbb Z_{2p}$ has a translation orbit of length $2p$, so such a translation must be zero. Hence each block orbit is full and covers each of these triangle orbits once.
\end{proof}

For the Type~II family, Grannell and Griggs~\cite[Section~3]{GrannellGriggsSurvey} stated the following direct construction in terms of difference quadruples without proof. Here we express the trapezoids by their vertex sets and include a proof for completeness.

\begin{lemma}\label{lem:type-II}
Let $p\ge5$ be prime. The base blocks $\{0,2i-1,2i+2j-1,4i+2j-2\}$ in $\mathbb Z_{2p}$, where $i$ and $j$ are integers with $1\le i\le\frac{p-3}{2}$ and $1\le j\le\frac{p-1}{2}-i$, represent full translation orbits covering every irregular triangle containing points of both parities exactly once and no other triangle.
\end{lemma}

\begin{proof}
Let $T$ be a triple in $\mathbb Z_{2p}$ containing points of both parities. Translate its unique vertex whose parity differs from the other two to $0$. Its translation orbit is then represented by an unordered pair of distinct odd residues in $\mathbb Z_{2p}$. The triangle is right-angled exactly when the pair contains $p$, and is isosceles exactly when the two residues are negatives of one another. All other pairs represent irregular triangles.

Set $r=2i-1$ and $s=2i+2j-1$. As $i$ and $j$ range over the stated integers, $(r,s)$ runs through all pairs of odd integers with $1\le r<s\le p-2$. The four triangles in $\{0,r,s,r+s\}$ represent the pairs
$$
\{r,s\},\quad\{-r,-s\},\quad\{-r,s\},\quad\{r,-s\}.
$$
Every unordered pair of distinct odd residues representing an irregular triangle has a unique expression $\{\varepsilon r,\eta s\}$ with $1\le r<s\le p-2$ odd and $\varepsilon,\eta\in\{1,-1\}$. Thus the displayed pairs represent all translation orbits of irregular triangles containing points of both parities, each exactly once. The four triangles in each base block lie in distinct translation orbits of length $2p$. As in the proof of Lemma~\ref{lem:type-I}, a translation preserving the block must fix each triangle, so the block orbit is full and the stated covering property follows.
\end{proof}

Types~I and~II have $(p-1)/2$ and $(p-1)(p-3)/8$ base blocks, respectively. Each Type~I base block contains three triangles of mixed parity, while each Type~II base block contains four. Together they cover
$$
\frac{3(p-1)}2+\frac{(p-1)(p-3)}2=\binom p2
$$
mixed-parity triangle translation orbits, as required. Only Type~III remains.

There are $2\binom p3$ triples in $\mathbb Z_{2p}$ whose vertices have the same parity. Among them, $p(p-1)$ are isosceles, since the apex has $2p$ choices and the nonzero even step is determined up to sign. None is right-angled. Thus the number of irregular triangles whose vertices have the same parity is
$$
2\binom p3-p(p-1)=\frac{p(p-1)(p-5)}3.
$$
Every block whose vertices have the same parity has a full translation orbit, since its stabilizer has order dividing $\gcd(4,2p)=2$ and translation by $p$ interchanges the two parity classes. In an exact Type~III covering, each orbit therefore accounts for $8p$ triangles, so $(p-1)(p-5)/24$ base blocks are required. The three base-block counts sum to
$$
\frac{p-1}{2}+\frac{(p-1)(p-3)}8+\frac{(p-1)(p-5)}{24}
=\frac{(2p-1)(2p-2)}{24},
$$
as required for a strictly cyclic $\operatorname{SQS}(2p)$.

\section{Auxiliary designs and affine normalization}\label{sec:auxiliary}

Let $p\ge5$ be prime. As in Section~\ref{sec:kohler-method}, place the elements $0,1,\ldots,p-1$ of $\mathbb F_p$ at the vertices of a regular $p$-gon in cyclic order. An isosceles triangle in $\mathbb F_p$ is a triple of the form $\{x-d,x,x+d\}$, where $x,d\in\mathbb F_p$ and $d\neq0$; all other triples in $\mathbb F_p$ are irregular. A trapezoid in $\mathbb F_p$ is a $4$-subset of the form $\{x,x+a,x+b,x+a+b\}$, where $x,a,b\in\mathbb F_p$. Thus the additive relations defining these figures are now interpreted in $\mathbb F_p$.

Every triangle in $\mathbb Z_{2p}$ whose vertices have the same parity can be translated to one with three even vertices. The bijection $2b\bmod 2p\mapsto b\bmod p$ ($b\in\mathbb Z$) identifies the even residues in $\mathbb Z_{2p}$ with the additive group of $\mathbb F_p$ and takes isosceles triangles to isosceles triangles. Two triangles with even vertices are translates in $\mathbb Z_{2p}$ if and only if their images are translates in $\mathbb F_p$. Hence the translation orbits of same-parity irregular triangles in $\mathbb Z_{2p}$ are in bijection with the translation orbits of irregular triangles in $\mathbb F_p$. We introduce the following auxiliary designs to describe the remaining covering problem on $\mathbb F_p$.

\begin{definition}
Let $p\ge5$ be prime. An $(\mathbb F_p,+)$-invariant family of $4$-subsets of $\mathbb F_p$, called \emph{blocks}, is denoted by $A(p)$ if every irregular triangle in $\mathbb F_p$ is contained in exactly one block of the family and no block contains an isosceles triangle. If every block of $A(p)$ is a trapezoid, the family is denoted by $A^*(p)$.
\end{definition}

Following the terminology for cyclic designs, we choose one representative from each translation orbit on $A(p)$ and call these representatives \emph{base blocks}.

Choose one base block from each translation orbit of an auxiliary design $A^*(p)$, and apply the inverse embedding $b\bmod p\mapsto2b\bmod 2p$ ($b\in\mathbb Z$) to every point of these blocks. Developing the resulting $4$-subsets under translations in $\mathbb Z_{2p}$ supplies the Type~III blocks. The number of irregular triangles in $\mathbb F_p$ is half the same-parity count in Section~\ref{sec:kohler-method}, namely $p(p-1)(p-5)/6$. Since $p\ge5$ is prime, every $4$-subset of $\mathbb F_p$ has a translation orbit of length $p$. Each block orbit in an auxiliary design $A(p)$ therefore covers $4p$ irregular triangles, so an auxiliary design $A(p)$ has $(p-1)(p-5)/24$ translation base blocks, in agreement with Section~\ref{sec:kohler-method}.

For an integer $v\ge4$, represent the elements of $\mathbb Z_v$ at the vertices of a regular $v$-gon. A triple $T$ in $\mathbb Z_v$ is called \emph{irregular} if its triangle is neither right-angled nor isosceles. A trapezoid $P$ in $\mathbb Z_v$ is a $4$-subset of the form $\{x,x+a,x+b,x+a+b\}$ for some $x,a,b\in\mathbb Z_v$; it is called \emph{admissible} if all four triples contained in it are irregular.

For a triple $T\subseteq\mathbb Z_v$ and a trapezoid $P\subseteq\mathbb Z_v$, write
$$
\mathcal T_{\pm}(T):=\{T+t,-T+t:t\in\mathbb Z_v\}, \qquad \mathcal P_{\pm}(P):=\{P+t,-P+t:t\in\mathbb Z_v\}
$$
for their orbits under translations and negation. For every trapezoid $P$, there exists $c\in\mathbb Z_v$ such that $-P+c=P$, so $\mathcal P_{\pm}(P)$ is also the translation orbit of $P$.

We use the following cyclic-group form of~\cite[Definition~2.1 and Lemma~2.2]{MunemasaSawa}.

\begin{definition}\label{def:full-kohler-graph}
Let $v\ge4$ be an integer. The \emph{full K\"ohler graph} $\mathcal K_v$ has as vertices the orbits $\mathcal T_{\pm}(T)$ of irregular triangles $T\subseteq\mathbb Z_v$. Its edges are the translation orbits $\mathcal P_{\pm}(P)$ of admissible trapezoids $P\subseteq\mathbb Z_v$. For $P=\{0,a,b,a+b\}$, the reflection $x\mapsto a+b-x$ pairs the four triples of $P$, and the edge $\mathcal P_{\pm}(P)$ joins $\mathcal T_{\pm}(\{0,a,b\})$ and $\mathcal T_{\pm}(\{0,a,a+b\})$.
\end{definition}

When $v=p$ is prime, we identify $\mathbb Z_p$ with the additive group of $\mathbb F_p$ and regard $\mathcal K_p$ as the graph on triangle and trapezoid orbits in $\mathbb F_p$.

For a prime $p>5$ with $p\equiv1$ or $5\pmod{12}$, the graph $\mathcal K_{2p}$ has two components: one consists of the orbits of irregular triangles containing points of both parities, and the other consists of those whose vertices have the same parity~\cite[Section~3]{GrannellGriggsSurvey}. The Type~II block orbits in Lemma~\ref{lem:type-II} form a $1$-factor of the first component. Choosing representatives with three even vertices and applying $2b\bmod 2p\mapsto b\bmod p$ gives an isomorphism from the second component to $\mathcal K_p$. Thus the remaining covering problem is equivalent to finding a $1$-factor (that is, a perfect matching) of $\mathcal K_p$.

\begin{lemma}[{\cite[Lemmas~2.3, 2.5 and~2.6]{MunemasaSawa}}]\label{lem:full-kohler-incidence}
Let $p\ge5$ be prime, and let $\mathcal K_p$ be the full K\"ohler graph in Definition~\ref{def:full-kohler-graph}. Every edge has two distinct endpoints, and the blocks in the translation orbit represented by that edge cover every triangle in either endpoint orbit exactly once. If $a,b\in\mathbb F_p$ and $\{0,a,b\}$ is irregular, the edges incident with $\mathcal T_{\pm}(\{0,a,b\})$ are the translation orbits of the admissible trapezoids among
$$
\{0,a,b,a+b\},\qquad \{0,a,b,b-a\},\qquad \{0,a,b,a-b\}.
$$
\end{lemma}

Let a $1$-factor of $\mathcal K_p$ be given, and take all blocks in the translation orbits represented by its edges. By Lemma~\ref{lem:full-kohler-incidence}, every irregular triangle in $\mathbb F_p$ occurs in exactly one of these blocks, so they form an auxiliary design $A^*(p)$. Conversely, every block orbit of an auxiliary design $A^*(p)$ is an edge of $\mathcal K_p$. Exact coverage, together with Lemma~\ref{lem:full-kohler-incidence}, implies that exactly one of these edges is incident with each vertex. Thus the block orbits form a $1$-factor of $\mathcal K_p$.

For primes $p\equiv1$ or $5\pmod{12}$, combining this correspondence with the Type~I and Type~II families gives K\"ohler's construction of a strictly $S$-cyclic $\operatorname{SQS}(2p)$ from a $1$-factor of $\mathcal K_p$; see~\cite[Corollary~VIII.10.16]{BethJungnickelLenz}.

Multiplication by a nonzero element of $\mathbb F_p$ preserves irregular triangles and admissible trapezoids, and hence acts on the vertices and edges of $\mathcal K_p$. Since negation is already included in their defining orbits, this action factors through $M_p:=\mathbb F_p^\times/\{\pm1\}$. Grouping vertices and edges under this action replaces their translation-and-negation orbits by affine orbits.

\begin{lemma}\label{lem:multiplier-quotient}
Let $p\ge5$ be prime, and let $M_p:=\mathbb F_p^\times/\{\pm1\}$ act on the full K\"ohler graph $\mathcal K_p$ by multiplication of the point labels. For an irregular triangle $T\subseteq\mathbb F_p$, the union of the triangle families represented by the vertices in the $M_p$-orbit of $\mathcal T_{\pm}(T)$ is its affine orbit
$$
\{mT+c:m\in\mathbb F_p^\times,\ c\in\mathbb F_p\}.
$$
For an admissible trapezoid $P\subseteq\mathbb F_p$, the union of the block families represented by the edges in the $M_p$-orbit of $\mathcal P_{\pm}(P)$ is its affine orbit
$$
\{mP+c:m\in\mathbb F_p^\times,\ c\in\mathbb F_p\}.
$$
\end{lemma}

\begin{proof}
Two vertices $\mathcal T_{\pm}(T)$ and $\mathcal T_{\pm}(T')$ lie in the same $M_p$-orbit if and only if $T'=mT+c$ for some $m\in\mathbb F_p^\times$ and $c\in\mathbb F_p$, since translations and negation are already included in $\mathcal T_{\pm}(T)$. The same argument applies to the block orbits $\mathcal P_{\pm}(P)$.
\end{proof}

The resulting orbit multigraph has one vertex for each affine orbit of irregular triangles and one edge for each affine orbit of admissible trapezoids. An edge may become a loop when its two endpoints belong to the same affine triangle orbit. Distinct affine trapezoid orbits remain distinct edges even when they have the same endpoints.

We describe these vertices and edges by normalizing triangle representatives to $\{0,1,a\}$ and trapezoid representatives to $\{0,1,a,1+a\}$. The following two subsections determine when different parameters represent the same affine orbit and which triangle orbits occur in each trapezoid orbit.

\subsection{Affine orbits of triangles}

Let $\operatorname{Aff}(\mathbb F_p)$ denote the group of maps of the form
$$
x\mapsto mx+c, \qquad m\in\mathbb F_p^\times,\quad c\in\mathbb F_p,
$$
under composition. This group acts on subsets of $\mathbb F_p$ by applying a map to each point. Every triple of distinct points in $\mathbb F_p$ is affinely equivalent to a triple $\{0,1,a\}$ with $a\in\mathbb F_p\setminus\{0,1\}$, where the value of $a$ depends on the ordering used to send two of the points to $0$ and $1$.

\begin{lemma}\label{lem:anharmonic-parameters}
Let $p$ be an odd prime, let $(x,y,z)$ be an ordering of three distinct points of $\mathbb F_p$, and set $a=(z-x)/(y-x)$. For each ordering of the three points, apply the affine map sending its first point to $0$ and its second point to $1$. The possible images of the third point form the set
$$
H(a):=\left\{a,a^{-1},1-a,(1-a)^{-1},\frac{a-1}{a},\frac{a}{a-1}\right\}.
$$
\end{lemma}

\begin{proof}
Since $a=(z-x)/(y-x)$, we have $z-x=a(y-x)$ and $z-y=(a-1)(y-x)$. For the six orderings $(x,y,z)$, $(x,z,y)$, $(y,x,z)$, $(y,z,x)$, $(z,x,y)$, and $(z,y,x)$, the affine map sending the first point to $0$ and the second to $1$ sends the third point to
$$
\frac{z-x}{y-x}=a,\quad
\frac{y-x}{z-x}=a^{-1},\quad
\frac{z-y}{x-y}=1-a,\quad
\frac{x-y}{z-y}=(1-a)^{-1},\quad
\frac{y-z}{x-z}=\frac{a-1}{a},\quad
\frac{x-z}{y-z}=\frac{a}{a-1},
$$
respectively. These six values form $H(a)$ and need not be distinct.
\end{proof}

For $a\in\mathbb F_p\setminus\{0,1\}$, let
$$
\mathcal T(a):=\operatorname{Orb}_{\operatorname{Aff}(\mathbb F_p)}(\{0,1,a\})
=\bigl\{\{c,m+c,ma+c\}:m\in\mathbb F_p^\times,\ c\in\mathbb F_p\bigr\}
$$
be the affine orbit of the triple $\{0,1,a\}$. It is also the union of $\mathcal T_{\pm}(\{0,m,ma\})$ over all $m\in\mathbb F_p^\times$. Thus $H(a)$ is a set of parameters, while $\mathcal T(a)$ is a family of triples in $\mathbb F_p$.

\begin{lemma}\label{lem:affine-triple-orbits}
Let $p$ be an odd prime, and let $a,b\in\mathbb F_p\setminus\{0,1\}$. Then $\mathcal T(a)=\mathcal T(b)$ if and only if $b\in H(a)$.
\end{lemma}

\begin{proof}
Suppose first that $\mathcal T(a)=\mathcal T(b)$. An affine map then sends $\{0,1,a\}$ to $\{0,1,b\}$. The preimages of $0$ and $1$, followed by the remaining point, give an ordering of $\{0,1,a\}$, so Lemma~\ref{lem:anharmonic-parameters} gives $b\in H(a)$. Conversely, if $b\in H(a)$, Lemma~\ref{lem:anharmonic-parameters} gives an ordering of $\{0,1,a\}$ for which the affine map sending the first two points to $0$ and $1$ sends the third point to $b$, and hence $\mathcal T(a)=\mathcal T(b)$.
\end{proof}

For $a\in\mathbb F_p\setminus\{0,1\}$, the triangle $\{0,1,a\}$ is isosceles if and only if one of its points is the average of the other two in $\mathbb F_p$. This gives $a\in\{-1,2,2^{-1}\}$, and
$$
H(-1)=H(2)=H(2^{-1})=\{-1,2,2^{-1}\}.
$$
Thus all isosceles triangles in $\mathbb F_p$ lie in a single affine orbit. Define
$$
X_p:=\mathbb F_p\setminus\{0,1,-1,2,2^{-1}\}.
$$
The parameters in $X_p$ are precisely those for which $\{0,1,a\}$ is irregular.

\begin{lemma}\label{lem:size-of-Ha}
Let $p$ be an odd prime, and let $a\in X_p$. Then $|H(a)|<6$ if and only if $a^2-a+1=0$, in which case $H(a)=\{a,a^{-1}\}$. An element $a\in X_p$ satisfying this equation exists if and only if $p\equiv1\pmod3$.
\end{lemma}

\begin{proof}
The six transformations defining $H(a)$ arise from the six permutations of the points of $\{0,1,a\}$. Two normalized values coincide if and only if a nonidentity permutation leaves the parameter $a$ unchanged. The three transpositions induce the transformations $a\mapsto a^{-1}$, $a\mapsto1-a$, and $a\mapsto a/(a-1)$, whose fixed points in $\mathbb F_p\setminus\{0,1\}$ are $-1$, $2^{-1}$, and $2$; all three values are excluded from $X_p$. The two $3$-cycles induce $a\mapsto(1-a)^{-1}$ and $a\mapsto(a-1)a^{-1}$, each of which fixes $a$ exactly when $a^2-a+1=0$. Hence, for $a\in X_p$, we have $|H(a)|<6$ if and only if $a^2-a+1=0$.

If $a^2-a+1=0$, then $a^{-1}=1-a$, and the six values in $H(a)$ reduce to $H(a)=\{a,a^{-1}\}$. Also $(a+1)(a^2-a+1)=a^3+1$, so $a^3=-1$. Since $a\in X_p$, we have $a\neq-1$, so $a$ has multiplicative order $6$ in $\mathbb F_p^\times$. Conversely, an element of multiplicative order $6$ satisfies $a^3=-1$ and $a\neq-1$, so it is a root of $a^2-a+1=0$ and lies in $X_p$. Since $\mathbb F_p^\times$ is cyclic, such an element exists if and only if $6\mid p-1$, or equivalently $p\equiv1\pmod3$.
\end{proof}

\subsection{Affine orbits of trapezoids}

Let $\{x,y,z,y+z-x\}$ be a trapezoid in $\mathbb F_p$. The affine map $t\mapsto(t-x)/(y-x)$ sends it to
$$
P(a):=\{0,1,a,1+a\},
$$
where $a=(z-x)/(y-x)\in\mathbb F_p\setminus\{0,1,-1\}$. For such a parameter $a$, let $\mathcal P(a):=\operatorname{Orb}_{\operatorname{Aff}(\mathbb F_p)}(P(a))$ denote the affine orbit of $P(a)$. Equivalently, $\mathcal P(a)$ is the union of the translation orbits $\mathcal P_{\pm}(mP(a))$ over all $m\in\mathbb F_p^\times$.

\begin{lemma}\label{lem:affine-trapezoid-orbit}
Let $p$ be an odd prime, and let $a\in\mathbb F_p\setminus\{0,1,-1\}$. As $(u,v)$ ranges over all ordered pairs of distinct points of $P(a)=\{0,1,a,1+a\}$, the image of $P(a)$ under the affine map sending $u$ to $0$ and $v$ to $1$ runs over the six $4$-subsets
$$
\begin{array}{lll}
\{0,1,a,1+a\},
&\{0,1,a^{-1},1+a^{-1}\},
&\{0,1,(1+a)^{-1},a(1+a)^{-1}\},\\
\{0,1,-a,1-a\},
&\{0,1,-a^{-1},(a-1)a^{-1}\},
&\{0,1,-(a-1)^{-1},a(a-1)^{-1}\},
\end{array}
$$
which are not necessarily distinct.
\end{lemma}

\begin{proof}
The reflection $t\mapsto1+a-t$ preserves $P(a)$ and pairs the twelve ordered pairs $(u,v)$ of distinct points of $P(a)$ into six orbits. Taking the representatives $(0,1)$, $(0,a)$, $(0,1+a)$, $(1,0)$, $(1,1+a)$, and $(1,a)$, the corresponding affine maps $t$, $t/a$, $t/(1+a)$, $1-t$, $(t-1)/a$, and $(t-1)/(a-1)$ carry $P(a)$ to the six listed $4$-subsets, respectively.
\end{proof}

\begin{lemma}\label{lem:affine-trapezoid-orbits}
Let $p$ be an odd prime, and let $a,b\in\mathbb F_p\setminus\{0,1,-1\}$. Then $\mathcal P(a)=\mathcal P(b)$ if and only if $b\in\{a,a^{-1},-a,-a^{-1}\}$. These four parameters are distinct unless $a^2=-1$, in which case they reduce to $\{a,-a\}$.
\end{lemma}

\begin{proof}
By Lemma~\ref{lem:affine-trapezoid-orbit}, the four sets $P(a)$, $P(a^{-1})$, $P(-a)$, and $P(-a^{-1})$ are the first, second, fourth, and fifth sets in the list, and hence all belong to $\mathcal P(a)$. Each of the four displayed values of $b$ therefore gives $\mathcal P(b)=\mathcal P(a)$.

Conversely, suppose that $\mathcal P(a)=\mathcal P(b)$. The set $P(b)=\{0,1,b,1+b\}$ then belongs to $\mathcal P(a)$ and contains both $0$ and $1$, so it must be one of the six $4$-subsets listed in Lemma~\ref{lem:affine-trapezoid-orbit}. The third and sixth sets cannot be of the form $P(b)$: in $P(b)$, the two points other than $0$ and $1$ differ by $\pm1$, whereas for the third set their difference is $(a-1)/(a+1)$ and for the sixth it is $(a+1)/(a-1)$, neither of which can equal $\pm1$ for $a\neq0$ in odd characteristic. Thus $b\in\{a,a^{-1},-a,-a^{-1}\}$.

Finally, since $a\notin\{0,1,-1\}$, a repetition among these four values can occur only when $a^{-1}=-a$, that is, when $a^2=-1$, in which case the set reduces to $\{a,-a\}$.
\end{proof}

Let $a\in\mathbb F_p\setminus\{0,1,-1\}$. The reflection $t\mapsto1+a-t$ pairs the four triples contained in $P(a)$. The triples $\{0,1,a\}$ and $\{1,a,1+a\}$ lie in the affine orbit $\mathcal T(a)$, while the other two triples $\{0,1,1+a\}$ and $\{0,a,1+a\}$ lie in $\mathcal T(-a)$, since $1-\{0,1,1+a\}=\{0,1,-a\}$. Consequently, if $\mathcal T(a)\neq\mathcal T(-a)$, every block in $\mathcal P(a)$ contains two triples from each of these two affine orbits; if $\mathcal T(a)=\mathcal T(-a)$, all four triples lie in the same affine orbit.

\begin{lemma}[{\cite[Lemma~2.2(ii)]{MunemasaSawa}, normalized form}]\label{lem:trapezoid-admissible}
Let $p$ be an odd prime, and let $a\in\mathbb F_p\setminus\{0,1,-1\}$. All four triples contained in $P(a)$ are irregular if and only if $a\notin\{0,\pm1,\pm2,\pm2^{-1}\}$.
\end{lemma}

\begin{proof}
The triples in $P(a)$ lie in $\mathcal T(a)$ or $\mathcal T(-a)$, with two triples represented by each parameter. Since the isosceles triangles form the affine orbit represented by $H(-1)=\{-1,2,2^{-1}\}$, all four triples in $P(a)$ are irregular if and only if $a,-a\in X_p$, which gives the stated set of excluded parameters.
\end{proof}

\begin{lemma}\label{lem:trapezoid-loop}
Let $p$ be an odd prime, and let $a\in\mathbb F_p\setminus\{0,1,-1\}$. Every triangle of $P(a)$ belongs to $\mathcal T(a)$ or $\mathcal T(-a)$, and $\mathcal T(a)=\mathcal T(-a)$ if and only if $a^2=-1$, $a^2-a-1=0$, or $a^2+a-1=0$.
\end{lemma}

\begin{proof}
By Lemma~\ref{lem:affine-triple-orbits}, the equality $\mathcal T(a)=\mathcal T(-a)$ is equivalent to $-a\in H(a)$. Comparing $-a$ with the six elements of $H(a)$ gives, since $a\neq0$, the three possibilities $-a=a^{-1}$, $-a=(1-a)^{-1}$, and $-a=(a-1)/a$, which are equivalent to $a^2=-1$, $a^2-a-1=0$, and $a^2+a-1=0$, respectively.
\end{proof}

By Lemma~\ref{lem:multiplier-quotient}, the multiplier quotient of $\mathcal K_p$ has the affine orbits $\mathcal T(a)$ of irregular triangles as its vertices and the affine orbits $\mathcal P(a)$ of admissible trapezoids as its edges. The edge represented by $\mathcal P(a)$ has endpoint orbits $\mathcal T(a)$ and $\mathcal T(-a)$, and becomes a loop precisely when $\mathcal T(a)=\mathcal T(-a)$.

For primes $p\equiv5\pmod{12}$, the graph $G_p$ in Definition~\ref{def:orbit-graph} is obtained from this multiplier quotient by deleting its loops and retaining parallel edges. Proposition~\ref{prop:p5-matching-reduction} shows that a $1$-factor of $G_p$ selects entire affine trapezoid orbits whose blocks form an auxiliary design $A^*(p)$.

For primes $p\equiv1\pmod{12}$, the two roots of $a^2-a+1=0$ represent one exceptional affine triangle orbit $\mathcal T(\zeta)$, where $\zeta^2-\zeta+1=0$. The vertices of $\mathcal K_p$ belonging to $\mathcal T(\zeta)$ have stabilizers of order $3$ in $M_p$, so selecting all blocks in an affine trapezoid orbit incident with $\mathcal T(\zeta)$ would cover each of its triples three times. The graph $G_p^*$ in Definition~\ref{def:p1-orbit-graph} is obtained from the multiplier quotient by deleting all loops together with the exceptional vertex $\mathcal T(\zeta)$ and its incident edges. Section~\ref{sec:p1} uses $1$-factors in specified vertex-deleted subgraphs of $G_p^*$, combined with alternating paths, the exceptional orbit $\mathcal T(\zeta)$, and the omitted loop orbits, to construct a $1$-factor of $\mathcal K_p$.

\section{The case \texorpdfstring{$p\equiv5\pmod{12}$}{p = 5 (mod 12)}}\label{sec:p5}

Throughout this section, let $p$ be a prime with $p\equiv5\pmod{12}$, put $X_p:=\mathbb F_p\setminus\{0,1,-1,2,2^{-1}\}$, and define the permutations $\sigma(a)=a^{-1}$ and $\theta(a)=1-a$ on $X_p$.

\begin{lemma}\label{lem:p5-free}
For every prime $p\equiv5\pmod{12}$ and every $a\in X_p$, the six maps generated by $\sigma$ and $\theta$ send $a$ to six distinct elements of $X_p$, so $|H(a)|=6$.
\end{lemma}

\begin{proof}
Since $p\equiv2\pmod3$, this follows immediately from Lemma~\ref{lem:size-of-Ha}.
\end{proof}

By Lemmas~\ref{lem:affine-triple-orbits} and~\ref{lem:p5-free}, each affine orbit of irregular triangles is represented by six parameters in $X_p$. Moreover, every irregular triangle in $\mathbb F_p$ has trivial setwise stabilizer in $\operatorname{Aff}(\mathbb F_p)$: if a nonidentity affine map preserved such a triple, it would induce a nontrivial permutation of its three points, and normalizing two of the points to $0$ and $1$ would produce two different orderings with the same value in $H(a)$, contradicting Lemma~\ref{lem:p5-free}. Since $|\operatorname{Aff}(\mathbb F_p)|=p(p-1)$ and there are $p(p-1)(p-5)/6$ irregular triangles in $\mathbb F_p$, there are $(p-5)/6$ affine orbits of irregular triangles, each containing $p(p-1)$ triangles.

\subsection{Affine orbit multigraph}

\begin{definition}\label{def:orbit-graph}
Let $p$ be a prime with $p\equiv5\pmod{12}$. The \emph{affine orbit multigraph} $G_p$ is the loopless multigraph with vertex set
$$
V(G_p)=\{\mathcal T(a):a\in X_p\},
\qquad
|V(G_p)|=\frac{p-5}{6}.
$$
Its edges are the distinct affine trapezoid orbits $\mathcal P(a)$ ($a\in X_p$) for which all four triples in $P(a)$ are irregular and $\mathcal T(a)\neq\mathcal T(-a)$. The edge $\mathcal P(a)$ joins $\mathcal T(a)$ and $\mathcal T(-a)$. Distinct affine trapezoid orbits are retained as parallel edges when they share the same endpoints, and vertex degrees count parallel edges with multiplicity.
\end{definition}

The underlying simple graph of $G_p$ is K\"ohler's affine graph~\cite[Section~3, Theorem~3]{GrannellGriggsSurvey}; retaining parallel edges simplifies the degree and bridge arguments below.

\begin{example}
Let $p=17$. We have
$$
H(3)=\{3,6,8,10,12,15\},
\qquad
H(4)=\{4,5,7,11,13,14\}.
$$
By Lemma~\ref{lem:affine-trapezoid-orbits}, the affine trapezoid orbits $\mathcal P(3)$ and $\mathcal P(5)$ are distinct since $5\notin\{3,3^{-1},-3,-3^{-1}\}=\{3,6,14,11\}$. On the other hand, $\mathcal P(3)$ joins $\mathcal T(3)$ and $\mathcal T(-3)=\mathcal T(14)=\mathcal T(4)$, while $\mathcal P(5)$ joins $\mathcal T(5)=\mathcal T(4)$ and $\mathcal T(-5)=\mathcal T(12)=\mathcal T(3)$. Thus $\mathcal P(3)$ and $\mathcal P(5)$ are parallel edges between $\mathcal T(3)$ and $\mathcal T(4)$.
\end{example}

\begin{lemma}\label{lem:p5-trapezoid-count}
Let $p>5$ be a prime with $p\equiv5\pmod{12}$, and let $\mathcal P(a)$ be an edge of $G_p$. Then $|\mathcal P(a)|=p(p-1)/2$, and every triple in $\mathcal T(a)\cup\mathcal T(-a)$ occurs in exactly one block of $\mathcal P(a)$.
\end{lemma}

\begin{proof}
Since $\mathcal T(a)\neq\mathcal T(-a)$, every affine map preserving $P(a)$ setwise preserves the pair of triples of $P(a)$ in $\mathcal T(a)$ and the pair in $\mathcal T(-a)$. The reflection $\tau_a(t)=1+a-t$ exchanges the two triples in each pair. For any $\psi\in\operatorname{Stab}_{\operatorname{Aff}(\mathbb F_p)}(P(a))$, either $\psi$ or $\tau_a\psi$ fixes a chosen triple of $P(a)$ in $\mathcal T(a)$ setwise, and is therefore the identity because an irregular triangle has trivial affine stabilizer. Hence $\operatorname{Stab}_{\operatorname{Aff}(\mathbb F_p)}(P(a))=\{1,\tau_a\}$ and $|\mathcal P(a)|=p(p-1)/2$.

Each of the $p(p-1)/2$ blocks in $\mathcal P(a)$ contains two triples from $\mathcal T(a)$. Since $|\mathcal T(a)|=p(p-1)$ and $\operatorname{Aff}(\mathbb F_p)$ acts transitively on $\mathcal T(a)$, every triple in $\mathcal T(a)$ occurs in exactly one block of $\mathcal P(a)$. The same holds for $\mathcal T(-a)$.
\end{proof}

\begin{proposition}[K\"ohler's affine reduction; cf.~{\cite[Section~3, Theorem~3]{GrannellGriggsSurvey}}]\label{prop:p5-matching-reduction}
Let $p$ be a prime with $p\equiv5\pmod{12}$. If $G_p$ has a $1$-factor, then there exists an auxiliary design $A^*(p)$.
\end{proposition}

\begin{proof}
For $p=5$, the empty family is an auxiliary design $A^*(5)$. For $p>5$, taking all blocks in the affine trapezoid orbits belonging to a $1$-factor of $G_p$ covers every irregular triangle of $\mathbb F_p$ exactly once by Lemma~\ref{lem:p5-trapezoid-count}, and no selected block contains an isosceles triangle by the definition of $G_p$. Thus the selected blocks form an auxiliary design $A^*(p)$.
\end{proof}

Since $|V(G_p)|=(p-5)/6$, a $1$-factor of $G_p$ has $(p-5)/12$ edges, yielding $\frac{p-5}{12}\cdot\frac{p(p-1)}{2}=\frac{p(p-1)(p-5)}{24}$ blocks in $A^*(p)$, or $(p-1)(p-5)/24$ translation base blocks, in agreement with Section~\ref{sec:auxiliary}.

\subsection{Tutte reduction}

Assume henceforth that $p>5$. For each $a\in X_p$, the six elements of $H(a)$ split into three inverse pairs:
$$
\{a,a^{-1}\},\qquad
\{1-a,(1-a)^{-1}\},\qquad
\left\{\frac{a-1}{a},\frac{a}{a-1}\right\}.
$$
By Lemma~\ref{lem:affine-trapezoid-orbits}, each inverse pair $\{u,u^{-1}\}\subseteq H(a)$ determines an affine trapezoid orbit $\mathcal P(u)=\mathcal P(u^{-1})$ joining $\mathcal T(a)$ and $\mathcal T(-u)$, unless $P(u)$ contains an isosceles triangle or $\mathcal T(-u)=\mathcal T(a)$.

Since $p\equiv5\pmod{12}$, we may fix a root $\mathrm{i}\in\mathbb F_p$ of $x^2=-1$; the orbit $\mathcal T(\mathrm{i})$ is independent of this choice since $-\mathrm{i}=\mathrm{i}^{-1}\in H(\mathrm{i})$. Whenever $x^2+x-1=0$ has a solution in $\mathbb F_p$, we also fix one of its roots and denote it by $\rho$; the other root is $-1-\rho\in H(\rho)$, so $\mathcal T(\rho)$ is likewise independent of the choice.

\begin{lemma}\label{lem:degrees-Gp}
Let $p>5$ be a prime with $p\equiv5\pmod{12}$. If $x^2+x-1=0$ has no solution in $\mathbb F_p$, then $\deg\mathcal T(-2)=\deg\mathcal T(\mathrm{i})=2$, and every other vertex of $G_p$ has degree $3$. If $x^2+x-1=0$ has a root $\rho\in\mathbb F_p$, then $\mathcal T(-2)$, $\mathcal T(\mathrm{i})$, and $\mathcal T(\rho)$ are pairwise distinct, $\deg\mathcal T(-2)=\deg\mathcal T(\mathrm{i})=2$, $\deg\mathcal T(\rho)=1$, and every other vertex of $G_p$ has degree $3$.
\end{lemma}

\begin{proof}
An inverse pair $\{u,u^{-1}\}$ at a vertex $\mathcal T(a)$ fails to give an edge of $G_p$ in two situations: either $\mathcal T(-u)$ is the affine orbit of isosceles triangles, or $\mathcal T(-u)=\mathcal T(u)$.

The first situation occurs when $-u\in H(-1)=\{-1,2,2^{-1}\}$, or equivalently $u\in\{1,-2,-2^{-1}\}$. Since $1\notin X_p$, the only possibility is $\{u,u^{-1}\}=\{-2,-2^{-1}\}$ at $\mathcal T(-2)$, where one of the three inverse pairs fails to give an edge.

By Lemma~\ref{lem:trapezoid-loop}, the second situation occurs when $u^2=-1$, $u^2-u-1=0$, or $u^2+u-1=0$. The two roots of $u^2=-1$ form the single self-negative inverse pair $\{\mathrm{i},-\mathrm{i}\}$ at $\mathcal T(\mathrm{i})$, which gives one omitted loop and leaves $\deg\mathcal T(\mathrm{i})=2$. When $\rho^2+\rho-1=0$, we have $\rho^{-1}=\rho+1$ and
$$
H(\rho)=\{\rho,\rho+1\}\cup\{-\rho,-\rho-1\}\cup\{1-\rho,\rho+2\},
$$
which contains the roots of both $u^2+u-1=0$ and $u^2-u-1=0$. The first two inverse pairs are negatives of one another and determine one omitted loop at $\mathcal T(\rho)$, while $\{1-\rho,\rho+2\}$ gives one edge, so $\deg\mathcal T(\rho)=1$.

It remains to verify that $\mathcal T(-2)$, $\mathcal T(\mathrm{i})$, and $\mathcal T(\rho)$ (when $\rho$ exists) are distinct. Direct substitution into $H(-2)$ shows that $\mathcal T(-2)=\mathcal T(\mathrm{i})$ can occur only in characteristics $5$ or $13$, while $\mathcal T(-2)=\mathcal T(\rho)$ can occur only in characteristics $5$ or $11$. Similarly, substituting the elements of $H(\mathrm{i})$ into $x^2+x-1$ shows that $\mathcal T(\mathrm{i})=\mathcal T(\rho)$ can occur only in characteristics $3$ or $5$. None of these primes satisfies $p>5$ and $p\equiv5\pmod{12}$.
\end{proof}

\begin{lemma}\label{lem:Gp-connected}
For every prime $p>5$ with $p\equiv5\pmod{12}$, the affine orbit multigraph $G_p$ is connected.
\end{lemma}

\begin{proof}
In the integer representatives $\{0,1,\ldots,p-1\}$, the excluded set $\mathbb F_p\setminus X_p=\{0,1,-1,2,2^{-1}\}$ consists of $0,1,2,(p+1)/2,p-1$. Since $a$ and $1-a$ represent the same triangle orbit, every vertex of $G_p$ can be written as $\mathcal T(k)$ for some integer $k$ with $3\le k\le(p-1)/2$. For $3\le k\le(p-3)/2$, both $k$ and $-k$ lie in $X_p$, and $\mathcal T(-k)=\mathcal T(k+1)$ because $-k=1-(k+1)\in H(k+1)$. Hence, if $\mathcal T(k)\neq\mathcal T(k+1)$, the trapezoid orbit $\mathcal P(k)$ gives an edge between them. Suppressing consecutive repetitions in $\mathcal T(3),\mathcal T(4),\ldots,\mathcal T((p-1)/2)$ therefore gives a walk containing every vertex of $G_p$.
\end{proof}

For a graph $G$ and $U\subseteq V(G)$, let $\operatorname{odd}(G-U)$ denote the number of odd components of $G-U$. For $L\subseteq V(G)$, let $\partial_G(L)$ (or $\partial(L)$) denote the set of edges with exactly one endpoint in $L$, with parallel edges counted separately. A component and its vertex set are denoted by the same symbol. Since $|V(G_p)|=(p-5)/6$ is even, Tutte's $1$-factor theorem gives a $1$-factor of $G_p$ if and only if $\operatorname{odd}(G_p-U)\le|U|$ for every $U\subseteq V(G_p)$. The following bridge condition is sufficient for $G_p$ to have a $1$-factor.

\begin{condition}[Bridge condition for $G_p$]\label{cond:bridge}
Let $p>5$ be a prime with $p\equiv5\pmod{12}$. For every bridge $e$ of $G_p$, if $x^2+x-1=0$ has no solution in $\mathbb F_p$, the two components of $G_p-e$ are not both of odd order; if $x^2+x-1=0$ has a root $\rho\in\mathbb F_p$, then whenever both components of $G_p-e$ have odd order, $\mathcal T(-2)$ and $\mathcal T(\rho)$ lie in different components.
\end{condition}

\begin{proposition}\label{prop:tutte-reduction}
Let $p>5$ be a prime with $p\equiv5\pmod{12}$. If $G_p$ satisfies Condition~\ref{cond:bridge}, then $G_p$ has a $1$-factor.
\end{proposition}

\begin{proof}
Let $U\subseteq V(G_p)$ and set $q:=\operatorname{odd}(G_p-U)$. By Tutte's $1$-factor theorem, it suffices to prove that $q\leq |U|$. This is immediate when $U=\varnothing$ or $U=V(G_p)$, so assume that $\varnothing\neq U\subsetneq V(G_p)$. Since $G_p$ is connected, every component of $G_p-U$ has at least one edge joining it to $U$.

Let $L$ be an odd component of $G_p-U$. Its boundary $\partial(L)$ consists of the edges joining $L$ to $U$. By Lemma~\ref{lem:degrees-Gp}, the only vertices of $G_p$ whose degrees differ from $3$ are $\mathcal T(-2)$ and $\mathcal T(\mathrm{i})$, both of degree $2$, and, when it exists, $\mathcal T(\rho)$ of degree $1$. The degree-sum formula gives $\sum_{x\in L}\deg_{G_p}(x)=2|E(L)|+|\partial(L)|$. Modulo $2$, the degree-$1$ vertex $\mathcal T(\rho)$ contributes the same parity as a degree-$3$ vertex, whereas each of $\mathcal T(-2)$ and $\mathcal T(\mathrm{i})$ changes the parity by $1$. Since $|L|$ is odd, we obtain
$$
|\partial(L)|\equiv 1-\bigl|L\cap\{\mathcal T(-2),\mathcal T(\mathrm{i})\}\bigr|\pmod 2.
$$
In particular, an odd component containing exactly one of $\mathcal T(-2)$ and $\mathcal T(\mathrm{i})$ has at least two boundary edges. Every other odd component has a positive odd number of boundary edges, hence either one or at least three.

Let $s$ be the number of odd components containing exactly one of $\mathcal T(-2)$ and $\mathcal T(\mathrm{i})$, and let $t$ be the number of odd components with $|\partial(L)|=1$. Summing the preceding lower bounds over all $q$ odd components of $G_p-U$ gives
$$
\sum_L|\partial(L)|\geq 3q-s-2t.
$$

The components counted by $s$ contain distinct vertices of $\{\mathcal T(-2),\allowbreak\mathcal T(\mathrm{i})\}\setminus U$. If a component $L$ counted by $t$ contains $\mathcal T(-2)$, then the parity relation above forces it to contain $\mathcal T(\mathrm{i})$ as well. Now let $L$ be a component counted by $t$ that does not contain $\mathcal T(-2)$, and denote its unique boundary edge by $e$. Then $e$ is a bridge of $G_p$. Since $|V(G_p)|$ is even and $|L|$ is odd, the other component of $G_p-e$ also has odd order. Condition~\ref{cond:bridge} therefore shows that $x^2+x-1=0$ must have a solution in $\mathbb F_p$ and that $L$ contains $\mathcal T(\rho)$.

Define $a_U:=\bigl|U\cap\{\mathcal T(-2),\mathcal T(\mathrm{i})\}\bigr|$. If $\mathcal T(\rho)$ exists, let $b_U:=1$ when $\mathcal T(\rho)\in U$ and $b_U:=0$ otherwise; if $\mathcal T(\rho)$ does not exist, set $b_U:=0$. Since distinct components of $G_p-U$ are disjoint, the preceding observations imply
$$
s+2t\leq 4-a_U-2b_U.
$$
Indeed, the two vertices $\mathcal T(-2)$ and $\mathcal T(\mathrm{i})$ outside $U$ can contribute at most $2-a_U$ to $s+2t$ (either up to two components counted by $s$, or one component counted by $t$ containing both when $a_U=0$), while $\mathcal T(\rho)$, when available outside $U$, can account for at most one component counted by $t$, contributing the remaining weight $2$. Hence
$$
\sum_L|\partial(L)|\geq 3q-4+a_U+2b_U.
$$

On the other hand, the boundary sets $\partial(L)$ of distinct odd components are pairwise disjoint, and every edge in $\partial(L)$ has its other endpoint in $U$. Summing the degrees in $G_p$ of the vertices in $U$ via Lemma~\ref{lem:degrees-Gp} therefore gives
$$
\sum_L|\partial(L)|\leq 3|U|-a_U-2b_U.
$$

Suppose toward a contradiction that $q>|U|$. Modulo $2$, we have $q\equiv|V(G_p-U)|=|V(G_p)|-|U|\equiv|U|\pmod2$ because $|V(G_p)|$ is even, so $q\geq |U|+2$. The last two inequalities then give
$$
3|U|+2+a_U+2b_U\leq\sum_L|\partial(L)|\leq 3|U|-a_U-2b_U,
$$
which is impossible. Therefore $\operatorname{odd}(G_p-U)\leq |U|$ for every $U\subseteq V(G_p)$, and Tutte's $1$-factor theorem implies that $G_p$ has a $1$-factor.
\end{proof}

\subsection{Bridge condition for \texorpdfstring{$G_p$}{G(p)}}

To verify Condition~\ref{cond:bridge} for $p>5$, we analyze the parameter set of a hypothetical odd bridge component violating the condition and show that it forces two arithmetic congruences on $p$; in Subsection~4.4, these congruences will be combined with a spectral bound on $\mathbb P^1(\mathbb F_p)$ to rule out all such components.

\begin{lemma}\label{lem:bridge-symmetric-difference}
Let $p>5$ be a prime with $p\equiv5\pmod{12}$. Suppose that Condition~\ref{cond:bridge} fails for a bridge $e$ of $G_p$, and let $L$ be the odd component of $G_p-e$ not containing $\mathcal T(-2)$ (and not containing $\mathcal T(\rho)$ when $\rho$ exists). Write $Y:=\bigcup_{\mathcal T(a)\in L}H(a)$. Then $Y$ is invariant under $\sigma(a)=a^{-1}$ and $\theta(a)=1-a$, and
$$
|Y|=6|L|,\qquad Y\cup(-Y)\subseteq X_p,\qquad |Y\mathbin{\triangle}(-Y)|=4.
$$
Moreover, $Y\cup(-Y)$ contains none of $0,\pm1,\pm2,\pm2^{-1},\pm\mathrm{i}$ and, when $\rho$ exists, none of $\rho,-1-\rho,-\rho,1+\rho$.
\end{lemma}

\begin{proof}
By Lemmas~\ref{lem:affine-triple-orbits} and~\ref{lem:p5-free}, the sets $H(a)$ corresponding to distinct vertices of $L$ are pairwise disjoint $6$-subsets of $X_p$. Their union $Y$ is therefore invariant under $\sigma$ and $\theta$, and $|Y|=6|L|$.

Since $e$ is the unique edge leaving $L$, we have $|\partial(L)|=1$. Since $|L|$ is odd and $\mathcal T(-2)\notin L$, the parity congruence in the proof of Proposition~\ref{prop:tutte-reduction} forces $\mathcal T(\mathrm{i})\notin L$. Thus $Y$ is disjoint from $H(-2)$, $H(\mathrm{i})$, and (when $\rho$ exists) $H(\rho)$. The only elements $y\in X_p$ for which $-y\notin X_p$ are $-2$ and $-2^{-1}$, both of which belong to $H(-2)$. Thus $-Y\subseteq X_p$, and hence $Y\cup(-Y)\subseteq X_p$. Because $\{\mathrm{i},-\mathrm{i}\}\subseteq H(\mathrm{i})$ and $\{\rho,-1-\rho,-\rho,1+\rho\}\subseteq H(\rho)$ are negation-invariant, none of these elements belongs to $Y\cup(-Y)$.

By the proof of Lemma~\ref{lem:degrees-Gp}, an inverse pair can fail to determine an edge of $G_p$ only at $\mathcal T(-2)$, $\mathcal T(\mathrm{i})$, or $\mathcal T(\rho)$. Since none of these vertices lies in $L$, every inverse pair in $Y$ determines an edge of $G_p$. Let $I=\{u,u^{-1}\}\subseteq Y$ be the inverse pair corresponding to the endpoint of the bridge $e$ in $L$. Since the other endpoint of $e$ lies outside $L$, its negative $-I=\{-u,-u^{-1}\}$ is disjoint from $Y$, so $I\subseteq Y\setminus(-Y)$ and $-I\subseteq(-Y)\setminus Y$. For every other inverse pair $J\subseteq Y$, the corresponding edge of $G_p$ has both endpoints in $L$, so $-J\subseteq Y$ and hence $J\subseteq Y\cap(-Y)$. Since the inverse pairs partition $Y$, we obtain $Y\setminus(-Y)=I$ and $(-Y)\setminus Y=-I$, and therefore $|Y\mathbin{\triangle}(-Y)|=4$.
\end{proof}

\begin{lemma}\label{lem:two-interval-form}
Let $p\ge5$ be an odd prime. Let $Z\subseteq\mathbb F_p$ be invariant under $t\mapsto t^{-1}$ and $t\mapsto1-t$, with $Z\cup(-Z)\subseteq X_p$, $|Z\mathbin{\triangle}(-Z)|=4$, and $|Z|=2\ell$ for an odd integer $\ell$. Then there exists $\xi\in\mathbb F_p$ such that
$$
Z=\{\xi+j:1\le j\le\ell\}\cup\{-\xi-j:0\le j<\ell\},
$$
where the two sets are disjoint. Putting $I_\xi:=\{\xi,\xi+1,\ldots,\xi+\ell\}$, we have $I_\xi\cap(-I_\xi)=\varnothing$, $Z\cup(-Z)=I_\xi\cup(-I_\xi)$, $(-Z)\setminus Z=\{\xi,-\xi-\ell\}$, and $\xi(\xi+\ell)=-1$.
\end{lemma}

\begin{proof}
Here an interval consists of consecutive elements in the cyclic order $0,1,\ldots,p-1,0$ on $\mathbb F_p$. Write $\ind_Z$ for the indicator function of $Z$. Since $Z$ is invariant under $t\mapsto1-t$, we have $\ind_Z(t+1)=\ind_Z(-t)$, and hence
$$
\sum_{t\in\mathbb F_p}|\ind_Z(t+1)-\ind_Z(t)|
=
\sum_{t\in\mathbb F_p}|\ind_Z(-t)-\ind_Z(t)|
=|Z\mathbin{\triangle}(-Z)|=4.
$$
The left-hand side counts the transitions between $Z$ and its complement in the cyclic order of $\mathbb F_p$. Each nonempty maximal interval contained in $Z$ contributes two transitions, so $Z$ consists of exactly two maximal intervals.

The involution $t\mapsto1-t$ permutes these two maximal intervals. It cannot preserve either interval: an invariant interval of odd length would contain the fixed point $2^{-1}$, while an invariant interval of even length would contain $\{0,1\}$, contradicting $Z\subseteq X_p$ in either case. Thus the two intervals are exchanged and have the same length $\ell$. Choosing $\xi$ so that the first interval is $\{\xi+1,\ldots,\xi+\ell\}$ gives the displayed form of $Z$.

Taking negatives shows that $Z\cup(-Z)=I_\xi\cup(-I_\xi)$. Since $|Z\cup(-Z)|=|Z|+\frac12|Z\mathbin{\triangle}(-Z)|=2\ell+2=|I_\xi|+|-I_\xi|$, the two sets $I_\xi$ and $-I_\xi$ are disjoint, and comparing the two descriptions gives $(-Z)\setminus Z=\{\xi,-\xi-\ell\}$. Because both $Z$ and $-Z$ are invariant under inversion and $\pm1\notin Z\cup(-Z)$, inversion exchanges $\xi$ and $-\xi-\ell$, giving $\xi(\xi+\ell)=-1$.
\end{proof}

Let $F_j$ and $\Lambda_j$ denote the Fibonacci and Lucas numbers, defined by $F_0=0$, $F_1=1$, $\Lambda_0=2$, $\Lambda_1=1$, and $u_{j+2}=u_{j+1}+u_j$ for $j\ge0$, so that $\Lambda_j=F_{j-1}+F_{j+1}$ for $j\ge1$. For every odd integer $m\ge3$, define
$$
\Pi_m:=\prod_{r=1}^{(m-1)/2}\bigl(r(m-r)-1\bigr).
$$

\begin{lemma}\label{lem:bridge-congruences}
Let $p>5$ be a prime, and let $\ell\ge3$ be an odd integer. Let $Z\subseteq\mathbb F_p$ be invariant under $t\mapsto t^{-1}$ and $t\mapsto1-t$, with $Z\cup(-Z)\subseteq X_p$, and suppose that $\xi\in\mathbb F_p$ satisfies
$$
Z=\{\xi+j:1\le j\le\ell\}\cup\{-\xi-j:0\le j<\ell\},
$$
$I_\xi\cap(-I_\xi)=\varnothing$ for $I_\xi:=\{\xi,\xi+1,\ldots,\xi+\ell\}$, and $\xi(\xi+\ell)=-1$. If $Z\cup(-Z)$ contains no root of $x^2=-1$ and no root of $x^2+x-1=0$, then
$$
p\mid\Lambda_\ell-\ell F_\ell,\qquad p\mid\Pi_\ell^2-1.
$$
\end{lemma}

\begin{proof}
The displayed interval representation gives $-Z=Z-1$. Regard $\varphi(z)=1/(z+1)$ as an element of $\operatorname{PGL}_2(\mathbb F_p)$ acting on $\mathbb P^1(\mathbb F_p)$. Since $z+1\in Z$ for every $z\in-Z$ and $Z\subseteq X_p$ is invariant under inversion, the restriction of $\varphi$ is a bijection from $-Z$ to $Z$. On the vertex set $Z\cup(-Z)$, draw a directed edge $z\to\varphi(z)$ for each $z\in-Z$. The two vertices $\xi$ and $-\xi-\ell$ in $(-Z)\setminus Z$ have indegree $0$ and outdegree $1$, the two vertices $-\xi$ and $\xi+\ell$ in $Z\setminus(-Z)$ have indegree $1$ and outdegree $0$, and all vertices in $Z\cap(-Z)$ have indegree and outdegree $1$. Thus the directed graph consists of two directed paths and some directed cycles.

Let $\iota(z):=-1/z$. Since $Z$ and $-Z$ are invariant under inversion, $\iota$ preserves $Z\cup(-Z)$ and satisfies $\iota\varphi=\varphi^{-1}\iota$, so it reverses every directed edge. A path or cycle preserved by $\iota$ would contain either a fixed point of $\iota$ (satisfying $z^2=-1$) or an edge reversed onto itself (giving $z=-2^{-1}$). Both are excluded by the hypotheses and by $Z\cup(-Z)\subseteq X_p$. Therefore $\iota$ exchanges the two paths and pairs the directed cycles.

The map $\iota$ sends the path beginning at $\xi$ to the other path, with its direction reversed. Since $\xi(\xi+\ell)=-1$, we have $\iota(\xi)=\xi+\ell$, so the other path ends at $\xi+\ell$. Therefore the path beginning at $\xi$ ends at $-\xi$; denote its length by $t$.

Since $p>5$, the fixed-point equation $z^2+z-1=0$ has two distinct roots over the algebraic closure of $\mathbb F_p$. Let $m$ be the length of a directed cycle through $z$. Then $\varphi^m$ fixes $z$ and the two fixed points of $\varphi$. The hypotheses exclude $z^2+z-1=0$, so these three points are distinct. Because a nonidentity fractional linear transformation cannot fix three distinct points, $\varphi^m$ is the identity, and the minimality of $m$ gives $m=\operatorname{ord}(\varphi)$, the order of $\varphi$ in $\operatorname{PGL}_2(\mathbb F_p)$.

The cycles occur in pairs, so counting the $2\ell$ directed edges gives $2\ell=2t+2k\operatorname{ord}(\varphi)$, or $\ell=t+k\operatorname{ord}(\varphi)$, for some $k\geq0$. Hence $\varphi^\ell(\xi)=\varphi^t(\xi)=-\xi$. Substituting this into
$$
\varphi^\ell(z)=\frac{F_{\ell-1}z+F_\ell}{F_\ell z+F_{\ell+1}}
$$
and using $\xi^2+1=-\ell\xi$ gives
$$
0=F_\ell(\xi^2+1)+\Lambda_\ell\xi
=(\Lambda_\ell-\ell F_\ell)\xi.
$$
Since $\xi\ne0$, we obtain $p\mid\Lambda_\ell-\ell F_\ell$.

For the second congruence, the relation $\xi(\xi+\ell)=-1$ gives
$$
(\xi+r)(\xi+\ell-r)=r(\ell-r)-1
$$
for $1\leq r\leq(\ell-1)/2$, and hence $\prod_{z\in I_\xi}z=-\Pi_\ell$. The disjoint union $I_\xi\cup(-I_\xi)=Z\cup(-Z)$ is invariant under inversion and contains none of $0,\pm1$, so the product of all its elements is $1$. Since $|I_\xi|=\ell+1$ is even, the products over $I_\xi$ and $-I_\xi$ are equal. Therefore $\Pi_\ell^2=1$ in $\mathbb F_p$.
\end{proof}

\subsection{A spectral bound for bridge components}

To bound the size of a bridge component, we first establish a spectral estimate on the projective line over $\mathbb F_p$ for an arbitrary odd prime $p$, which will also be used in Section~\ref{sec:p1}.

Let $p$ be an odd prime, and write $\Omega:=\mathbb P^1(\mathbb F_p)=\mathbb F_p\cup\{\infty\}$. We extend the maps $\sigma(a)=a^{-1}$ and $\theta(a)=1-a$, together with negation and the translations $a\mapsto a\pm1$, to permutations of $\Omega$ by setting
$$
0^{-1}=\infty,\qquad \infty^{-1}=0,\qquad -\infty=\infty,\qquad 1-\infty=\infty,\qquad \infty\pm1=\infty.
$$
Let $\Gamma_p$ be the Schreier multigraph on $\Omega$ determined by the three maps $a\mapsto a+1$, $a\mapsto a-1$, and $a\mapsto-a^{-1}$. For real-valued functions on $\Omega$, we equip $\mathbb R^\Omega$ with the counting inner product $\langle f,g\rangle:=\sum_{a\in\Omega}f(a)g(a)$. The unnormalized Laplacian of $\Gamma_p$ is the self-adjoint operator $\Delta$ defined by
$$
(\Delta f)(a):=3f(a)-f(a+1)-f(a-1)-f(-a^{-1}).
$$

\begin{lemma}[{\cite[Theorem~4.4.2(ii) and the following paragraph, p.~53]{Lubotzky}}]\label{lem:projective-spectral-gap}
Let $p$ be an odd prime. For every real-valued function $f$ on $\Omega$ with $\sum_{z\in\Omega}f(z)=0$, one has
$$
\langle\Delta f,f\rangle\geq0.003\,\langle f,f\rangle.
$$
\end{lemma}

\begin{proof}
Let $\Gamma=\operatorname{PSL}_2(\mathbb Z)$, and let $T,S\in\Gamma$ be represented by
$$
T=\begin{pmatrix}1&1\\0&1\end{pmatrix},\qquad
S=\begin{pmatrix}0&1\\-1&0\end{pmatrix}.
$$
Reduction modulo $p$ gives an action of $\Gamma$ on $\Omega$ by fractional linear transformations, with $Tz=z+1$ and $Sz=-z^{-1}$. Since $S^2=1$ in $\Gamma$, the generating set $\{T,T^{-1},S\}$ is inverse-closed.

Let $H_p$ be the image in $\Gamma$ of the congruence subgroup
$$
\Gamma_0(p)=\left\{
\begin{pmatrix}a&b\\c&d\end{pmatrix}\in\operatorname{SL}_2(\mathbb Z):c\equiv0\pmod p
\right\}.
$$
A matrix fixes $\infty$ modulo $p$ exactly when $c\equiv0\pmod p$, so $H_p$ is the stabilizer of $\infty$ under this action. The action is transitive, since $S\infty=0$ and $T^m0=m\bmod p$ for $0\leq m<p$. Hence $gH_p\mapsto g\infty$ is an equivariant bijection from $\Gamma/H_p$ to $\Omega$, and $H_p$ contains the image of the principal congruence subgroup of level $p$.

For the permutation representation $\pi(g)f(z)=f(g^{-1}z)$ on $\Omega$, we have
$$
3I-\pi(T)-\pi(T^{-1})-\pi(S)=\Delta,
$$
where the three generator contributions are counted separately even when their images coincide. Thus $\Delta$ is the unnormalized Laplacian in~\cite[Definition~4.2.1 and Proposition~4.2.2]{Lubotzky} for the generating set $\{T,T^{-1},S\}$.

The Schreier multigraph $\Gamma_p$ coincides with the graph in~\cite[Theorem~4.4.2(ii), p.~53]{Lubotzky}, which uses precisely $z\mapsto z+1$, $z\mapsto z-1$, and $z\mapsto-z^{-1}$, and the paragraph following its proof gives $0.003$ as a lower bound for the first positive eigenvalue of this unnormalized Laplacian. Transitivity implies that $\ker\Delta$ consists of the constant functions, giving the claimed inequality on their orthogonal complement.
\end{proof}

\begin{lemma}\label{lem:parameter-energy}
Let $p$ be an odd prime. If $Z\subseteq\Omega$ is invariant under $\sigma$ and $\theta$, then its indicator function $\ind_Z$ satisfies
$$
\left\langle\Delta\ind_Z,\ind_Z\right\rangle=\frac32\left|Z\mathbin{\triangle}(-Z)\right|.
$$
\end{lemma}

\begin{proof}
For $f:=\ind_Z$, we have the identities $a+1=\theta(-a)$, $a-1=-\theta(a)$, and $-a^{-1}=-\sigma(a)$ on $\Omega$. Since $f$ is invariant under $\sigma$ and $\theta$, we have $f(a+1)=f(\theta(-a))=f(-a)$, and the changes of variable $b=\theta(a)$ and $b=\sigma(a)$ give
$$
\sum_{a\in\Omega}f(a)f(a+1)=\sum_{a\in\Omega}f(a)f(a-1)=\sum_{a\in\Omega}f(a)f(-a^{-1})=\sum_{a\in\Omega}f(a)f(-a)=|Z\cap(-Z)|.
$$
Therefore, $\left\langle\Delta f,f\right\rangle=3|Z|-3|Z\cap(-Z)|=\frac32\left|Z\mathbin{\triangle}(-Z)\right|$.
\end{proof}

\begin{lemma}\label{lem:spectral-cutoff}
Let $p$ be an odd prime. If $Z\subseteq\Omega$ is invariant under $\sigma$ and $\theta$, then
$$
0.003\,|Z|\left(1-\frac{|Z|}{p+1}\right)\leq\frac32\left|Z\mathbin{\triangle}(-Z)\right|.
$$
Consequently, $\min\{|Z|,p+1-|Z|\}\leq1000|Z\mathbin{\triangle}(-Z)|$. In particular, if $|Z\mathbin{\triangle}(-Z)|=4$, then $\min\{|Z|,p+1-|Z|\}\leq4000$.
\end{lemma}

\begin{proof}
Let $f:=\ind_Z$ and set $f_0:=f-\frac{|Z|}{p+1}\ind_\Omega$. Then $\sum_{a\in\Omega}f_0(a)=0$, $\Delta\ind_\Omega=0$, and $\langle f_0,f_0\rangle = |Z|\bigl(1-\frac{|Z|}{p+1}\bigr)$. Lemmas~\ref{lem:projective-spectral-gap} and~\ref{lem:parameter-energy} give
$$
0.003\,\langle f_0,f_0\rangle\leq\langle\Delta f_0,f_0\rangle=\langle\Delta f,f\rangle=\frac32\left|Z\mathbin{\triangle}(-Z)\right|,
$$
proving the first inequality. Putting $u:=\min\{|Z|,p+1-|Z|\}\le(p+1)/2$, we have $|Z|\bigl(1-\frac{|Z|}{p+1}\bigr)=\frac{u(p+1-u)}{p+1}\ge\frac u2$, which gives the remaining assertions.
\end{proof}

We now return to a prime $p>5$ with $p\equiv5\pmod{12}$ and the bridge component $L$ of $G_p-e$ in Lemma~\ref{lem:bridge-symmetric-difference}. Since $|V(G_p)|=(p-5)/6$ and both components of $G_p-e$ have odd order, the number of vertices in the other component is a positive odd integer $d:=(p-5)/6-|L|$.

\begin{corollary}\label{cor:p5-component-cutoff}
Under the hypotheses of Lemma~\ref{lem:bridge-symmetric-difference}, either $|L|\le665$ or $d\le665$.
\end{corollary}

\begin{proof}
By Lemma~\ref{lem:bridge-symmetric-difference}, the set $Y=\bigcup_{\mathcal T(a)\in L}H(a)\subseteq X_p$ is invariant under $\sigma$ and $\theta$ on $\Omega$, with $|Y|=6|L|$ and $|Y\mathbin{\triangle}(-Y)|=4$. Its complement $\Omega\setminus Y$ consists of the six elements of $\Omega\setminus X_p=\{\infty,0,1,-1,2,2^{-1}\}$ together with the $6d$ parameters in $X_p$ corresponding to the $d$ vertices in the other component of $G_p-e$. Hence $p+1-|Y|=6(d+1)$. Applying Lemma~\ref{lem:spectral-cutoff} to $Z=Y$ gives $\min\{6|L|,6(d+1)\}\leq4000$, so $|L|\leq666$ or $d\leq665$. Since $|L|$ is odd, the first bound improves to $|L|\leq665$.
\end{proof}

Setting $\ell:=3|L|$, we see from the fact that $|L|$ and $d$ are positive odd integers that $\ell\equiv3\pmod6$ and $p=2\ell+6d+5\ge2\ell+11$. By Lemma~\ref{lem:bridge-symmetric-difference}, $Y\cup(-Y)$ contains neither the roots $\pm\mathrm{i}$ of $x^2=-1$ nor (when they exist) the roots of $x^2+x-1=0$. Hence Lemma~\ref{lem:bridge-congruences} applies to $Z=Y$ and shows that $p$ divides both $\Lambda_\ell-\ell F_\ell$ and $\Pi_\ell^2-1$.

\begin{lemma}\label{lem:p5-certificate-small}
For every $\ell\in\{3,9,\ldots,1995\}$, there is no prime $p\equiv5\pmod{12}$ with $p\ge2\ell+11$ that divides both $\Lambda_\ell-\ell F_\ell$ and $\Pi_\ell^2-1$.
\end{lemma}

\begin{proof}
For each of the $333$ values $\ell\in\{3,9,\ldots,1995\}$, complete factorization of $\gcd(\Pi_\ell^2-1,|\Lambda_\ell-\ell F_\ell|)$ yields no prime factor $p\equiv5\pmod{12}$ satisfying $p\ge2\ell+11$; see~\cite{LiuFengZenodo}.
\end{proof}

Suppose now that $d\leq665$. Using the integer representatives $0,1,\ldots,p-1$, choose $I_\xi$ or $-I_\xi$ from Lemma~\ref{lem:two-interval-form} so that its least element $u$ is the least element of $I_\xi\cup(-I_\xi)$. Then the two intervals are $\{u,u+1,\ldots,u+\ell\}$ and $\{p-u-\ell,\ldots,p-u\}$, with $u+\ell<p-u-\ell$. Since $1\notin Y\cup(-Y)$, we have $u\geq2$, and $p=2\ell+6d+5$ gives $u\leq3d+2$. By Lemma~\ref{lem:two-interval-form}, $u(u+\ell)\equiv-1\pmod p$; substituting $2\ell=p-6d-5$ yields
$$
p\mid N_{d,u}:=(6d+5)u-2u^2-2.
$$
If $u=3d+2$, then $N_{d,u}=3d$, which is impossible since $0<3d<p$. Hence $2\leq u\leq3d+1$. On this interval, the concave quadratic $N_{d,u}$ takes endpoint values $12d$ and $9d+1$, and completing the square for $d\leq665$ gives $0<N_{d,u}\leq1\,995\,001$.

\begin{lemma}\label{lem:p5-certificate-complement}
Let $d$ and $u$ be integers with $d$ odd, $1\le d\le665$, and $2\le u\le3d+1$, and let $p\equiv5\pmod{12}$ be a prime divisor of $N_{d,u}=(6d+5)u-2u^2-2$. If $\ell:=(p-6d-5)/2$ satisfies $\ell>0$ and $\ell\equiv3\pmod6$, then $\Lambda_\ell-\ell F_\ell\not\equiv0\pmod p$ or $\Pi_\ell^2\not\equiv1\pmod p$.
\end{lemma}

\begin{proof}
We enumerate all pairs $(d,u)$ with $d\in\{1,3,\ldots,665\}$ and $2\le u\le3d+1$, and completely factor $N_{d,u}$. For each prime divisor $p\equiv5\pmod{12}$, we set $\ell=(p-6d-5)/2$ and retain $(p,\ell,d)$ when $\ell>0$ and $\ell\equiv3\pmod6$; here $\ell$ is determined by $p$ and $d$, rather than restricted to $\ell\le1995$. Recording each distinct triple $(p,\ell,d)$ once yields $42\,417$ triples.

Computing $F_\ell$ and $\Lambda_\ell$ modulo $p$ shows that only two retained triples satisfy $\Lambda_\ell-\ell F_\ell\equiv0\pmod p$:
$$
\begin{array}{r|r|r|r}
p & \ell & d & \Pi_\ell^2 \bmod p\\
\hline
6089 & 1269 & 591 & 5834\\
10\,601 & 3873 & 475 & 1929
\end{array}
$$
In both cases $\Pi_\ell^2\not\equiv1\pmod p$; see~\cite{LiuFengZenodo}.
\end{proof}

By Corollary~\ref{cor:p5-component-cutoff}, either $|L|\le665$ or $d\le665$, and Lemmas~\ref{lem:p5-certificate-small} and~\ref{lem:p5-certificate-complement} contradict Lemma~\ref{lem:bridge-congruences} in both cases. Thus $G_p$ satisfies Condition~\ref{cond:bridge}. By Propositions~\ref{prop:p5-matching-reduction} and~\ref{prop:tutte-reduction} (together with $A^*(5)=\varnothing$), an auxiliary design $A^*(p)$ exists for every prime $p\equiv5\pmod{12}$.

\section{The case \texorpdfstring{$p\equiv1\pmod{12}$}{p = 1 (mod 12)}}\label{sec:p1}

Throughout this section, let $p$ be a prime with $p\equiv1\pmod{12}$, and write $M_p:=\mathbb F_p^\times/\{\pm1\}$. By Lemma~\ref{lem:multiplier-quotient}, the $M_p$-orbits on the vertices and edges of the full K\"ohler graph $\mathcal K_p$ correspond to affine orbits of irregular triangles and admissible trapezoids, respectively. Unlike the case $p\equiv5\pmod{12}$, the action of $M_p$ on $V(\mathcal K_p)$ has one orbit whose vertex stabilizers have order $3$.

\subsection{Affine orbit multigraph}\label{subsec:p1-orbit}

Since $p\equiv1\pmod{12}$, the equations $a^2-a+1=0$ and $a^2=-1$ each have two roots in $\mathbb F_p$. Throughout this section, we fix roots $\zeta,\mathrm{i}\in\mathbb F_p$ satisfying $\zeta^2-\zeta+1=0$ and $\mathrm{i}^2=-1$. Whenever $x^2+x-1=0$ has a solution in $\mathbb F_p$, we also fix one of its roots and denote it by $\rho$.

By Lemma~\ref{lem:size-of-Ha}, $H(\zeta)=\{\zeta,\zeta^{-1}\}$, where $\zeta^{-1}=1-\zeta$. Thus $\mathcal T(\zeta)$ is the unique affine orbit of irregular triangles represented by only two normalized parameters; every other affine orbit of irregular triangles is represented by six parameters.

The setwise affine stabilizer of $\{0,1,\zeta\}$ is
$$
\operatorname{Stab}_{\operatorname{Aff}(\mathbb F_p)}(\{0,1,\zeta\})
=\bigl\{x\mapsto x,\quad x\mapsto1+(\zeta-1)x,\quad x\mapsto\zeta-\zeta x\bigr\}.
$$
Using $\zeta^2=\zeta-1$, the last two maps induce the cycles $(0\,1\,\zeta)$ and $(0\,\zeta\,1)$ on the triple. An affine map is determined by its values at two distinct points, so the stabilizer acts faithfully on these three points. Any further element would therefore induce a transposition. An affine map exchanging two points and fixing the third has the form $x\mapsto-x+c$, and its fixed point is the average of the two exchanged points in $\mathbb F_p$, which would make $\{0,1,\zeta\}$ an isosceles triangle. Hence the stabilizer has order $3$ and $|\mathcal T(\zeta)|=p(p-1)/3$, whereas every other affine orbit of irregular triangles has size $p(p-1)$. Since each vertex $\mathcal T_{\pm}(T)$ of $\mathcal K_p$ contains $2p$ triples, $\mathcal T(\zeta)$ is the union of $(p-1)/6$ vertices of $\mathcal K_p$, whereas every other affine orbit of irregular triangles is the union of $(p-1)/2$ vertices.

Selecting all blocks in the affine trapezoid orbit joining $\mathcal T(\zeta)$ and $\mathcal T(-\zeta)$ would cover each triple in $\mathcal T(\zeta)$ three times, as shown in Lemma~\ref{lem:p1-branch-lift} below. Thus a $1$-factor of $\mathcal K_p$ cannot be obtained by selecting entire affine trapezoid orbits incident with $\mathcal T(\zeta)$.

Define
$$
X_p:=\mathbb F_p\setminus\{0,1,-1,2,2^{-1}\},\qquad
X_p^*:=X_p\setminus\{\zeta,\zeta^{-1}\}.
$$
Thus $a\in X_p^*$ if and only if $\{0,1,a\}$ is irregular and $\mathcal T(a)\neq\mathcal T(\zeta)$.

\begin{definition}\label{def:p1-orbit-graph}
Let $p$ be a prime with $p\equiv1\pmod{12}$. The \emph{affine orbit multigraph} $G_p^*$ is the loopless multigraph with vertex set $V(G_p^*)=\{\mathcal T(a):a\in X_p^*\}$. Its edges are the distinct affine trapezoid orbits $\mathcal P(a)$, with $a\in X_p^*$, for which all four triples of $P(a)$ are irregular, $\mathcal T(-a)\in V(G_p^*)$, and $\mathcal T(a)\neq\mathcal T(-a)$. The edge $\mathcal P(a)$ joins $\mathcal T(a)$ and $\mathcal T(-a)$. Distinct affine trapezoid orbits are retained as distinct edges when they have the same endpoints.
\end{definition}

The total number of irregular triangles in $\mathbb F_p$ is $p(p-1)(p-5)/6$. Removing the exceptional orbit $\mathcal T(\zeta)$ of size $p(p-1)/3$ and dividing by $p(p-1)$ gives $|V(G_p^*)|=(p-7)/6$, which is odd.

As in Section~\ref{sec:p5}, the three inverse pairs in $H(a)$ describe the possible incidences at a vertex $\mathcal T(a)$ of $G_p^*$. The inverse pairs that fail to give edges of $G_p^*$ are as follows:
\begin{itemize}
\item The inverse pair $\{-2,-2^{-1}\}\subseteq H(-2)$ leads to the affine orbit of isosceles triangles and gives no admissible edge.
\item The inverse pair $\{\mathrm{i},-\mathrm{i}\}\subseteq H(\mathrm{i})$ gives an affine trapezoid orbit with $\mathcal T(\mathrm{i})$ at both ends, which is omitted as a loop.
\item The inverse pair $\{-\zeta,-\zeta^{-1}\}\subseteq H(-\zeta)$ has negative $H(\zeta)$, so the corresponding affine trapezoid orbit joins $\mathcal T(-\zeta)$ to the removed vertex $\mathcal T(\zeta)$.
\item When $x^2+x-1=0$ has a root $\rho\in\mathbb F_p$, the roots of $x^2+x-1=0$ and $x^2-x-1=0$ lie in $H(\rho)$ and represent the same vertex $\mathcal T(\rho)$. Since $\rho^{-1}=\rho+1$, the three inverse pairs in $H(\rho)$ are $\{\rho,\rho+1\}$, $\{-\rho,-\rho-1\}$, and $\{1-\rho,\rho+2\}$. The first two are negatives of one another and determine one omitted loop at $\mathcal T(\rho)$. The negative of the third pair is $\{\rho-1,-\rho-2\}$ with $(\rho-1)^{-1}=-\rho-2$, so the unique edge of $G_p^*$ incident with $\mathcal T(\rho)$ joins it to $\mathcal T(\rho-1)$.
\end{itemize}

\begin{lemma}\label{lem:p1-degrees}
Let $p>13$ be a prime with $p\equiv1\pmod{12}$. If $x^2+x-1=0$ has no solution in $\mathbb F_p$, then $\mathcal T(-2)$, $\mathcal T(\mathrm{i})$, and $\mathcal T(-\zeta)$ are pairwise distinct, each has degree $2$ in $G_p^*$, and every other vertex has degree $3$. If $x^2+x-1=0$ has a root $\rho\in\mathbb F_p$, then $\mathcal T(-2)$, $\mathcal T(\mathrm{i})$, $\mathcal T(-\zeta)$, $\mathcal T(\rho)$, and $\mathcal T(\rho-1)$ are pairwise distinct; in this case, the first three vertices have degree $2$, $\mathcal T(\rho)$ has degree $1$, and every other vertex (in particular, $\mathcal T(\rho-1)$) has degree $3$.
\end{lemma}

\begin{proof}
Lemmas~\ref{lem:affine-trapezoid-orbits}, \ref{lem:trapezoid-admissible}, and~\ref{lem:trapezoid-loop}, together with the preceding discussion of inverse pairs, yield the stated degrees provided that the listed vertices are pairwise distinct.

The parameters $-2$, $\mathrm{i}$, $-\zeta$ and, when $\rho$ exists, $\rho$ and $\rho-1$ are roots of the polynomials
$$
f_1=x+2,\quad f_2=x^2+1,\quad f_3=x^2+x+1,
\quad f_4=x^2+x-1,\quad f_5=x^2+3x+1,
$$
respectively. For $1\le i<j\le 5$, a coincidence between the corresponding affine triangle orbits gives some $a\in X_p$ with $f_i(a)=0$ and $f_j(h(a))=0$, where $h(x)=N(x)/D(x)$ is one of the six fractional linear transformations defining $H(x)$. Since $a\in X_p$, we have $D(a)\ne0$, so $p$ divides the nonzero integer $\operatorname{Res}_x(f_i,D^2f_j(N/D))$. Computing these resultants across all six transformations gives the following prime divisors:
$$
\begin{array}{c|l@{\qquad}c|l}
(i,j) & \text{primes} & (i,j) & \text{primes}\\
\hline
(1,2) & 2,5,13 & (1,3) & 3,13,19\\
(1,4) & 5,11 & (1,5) & 19,31\\
(2,3) & 13 & (2,4) & 3,5\\
(2,5) & 3,41 & (3,4) & 2\\
(3,5) & 2,19 & (4,5) & 2,3
\end{array}
$$
The only prime in this table congruent to $1$ modulo $12$ is $13$. Hence no coincidence occurs for $p>13$. All $60$ integer resultants, together with their Sylvester matrices and an exact verification script, are included in~\cite{LiuFengZenodo}.
\end{proof}

\begin{lemma}\label{lem:p1-connected}
For every prime $p\equiv1\pmod{12}$, the affine orbit multigraph $G_p^*$ is connected.
\end{lemma}

\begin{proof}
As in Lemma~\ref{lem:Gp-connected}, every affine orbit of irregular triangles has a representative among $\mathcal T(3),\allowbreak\mathcal T(4),\allowbreak\ldots,\allowbreak\mathcal T((p-1)/2)$, and consecutive distinct terms other than $\mathcal T(\zeta)$ are joined in $G_p^*$ because $\mathcal T(-k)=\mathcal T(k+1)$. Let $z$ be the representative of $\mathcal T(\zeta)$ in this sequence. Since $z^2-z+1=0$, we have $(z-1)^{-1}=-z$, and hence $\mathcal T(z-1)=\mathcal T(-z)=\mathcal T(z+1)$. Thus deleting the exceptional term $\mathcal T(z)$ and suppressing consecutive repetitions gives a walk in $G_p^*$ containing every vertex.
\end{proof}

\subsection{Tutte reduction}\label{subsec:p1-tutte}

Assume throughout this subsection that $p>13$. Since $|V(G_p^*)|=(p-7)/6$ is odd, $G_p^*$ itself has no $1$-factor; instead, we use the following bridge condition to establish the existence of $1$-factors in specific vertex-deleted subgraphs of $G_p^*$, which will be lifted to $\mathcal K_p$ in Subsection~\ref{subsec:p1-lift}.

Let $e$ be a bridge of $G_p^*$ and let $L$ be the vertex set of the odd component of $G_p^*-e$. By Lemma~\ref{lem:p1-degrees}, the three degree-$2$ vertices of $G_p^*$ are $\mathcal T(-2)$, $\mathcal T(\mathrm{i})$, and $\mathcal T(-\zeta)$. Relative to degree $3$, each of these changes the degree sum by $1$, while the possible degree-$1$ vertex $\mathcal T(\rho)$ changes it by $2$ and does not affect the parity. Since $|L|$ is odd and $|\partial(L)|=1$, the degree-sum formula gives
$$
\bigl|L\cap\{\mathcal T(-2),\mathcal T(\mathrm{i}),\mathcal T(-\zeta)\}\bigr|\equiv0\pmod2.
$$
Thus $L$ contains either none or exactly two of $\mathcal T(-2)$, $\mathcal T(\mathrm{i})$, and $\mathcal T(-\zeta)$.

\begin{condition}[Bridge condition for $G_p^*$]\label{cond:p1-bridge}
Let $p>13$ be a prime with $p\equiv1\pmod{12}$. For every bridge $e$ of $G_p^*$, let $L$ be the vertex set of the odd component of $G_p^*-e$. If $x^2+x-1=0$ has no solution in $\mathbb F_p$, then $L$ contains exactly two of $\mathcal T(-2)$, $\mathcal T(\mathrm{i})$, and $\mathcal T(-\zeta)$. If $x^2+x-1=0$ has a root $\rho\in\mathbb F_p$, then either $L$ contains exactly two of these three vertices or $L=\{\mathcal T(\rho)\}$.
\end{condition}

Suppose that $G_p^*$ satisfies Condition~\ref{cond:p1-bridge}.

If $x^2+x-1=0$ has no solution in $\mathbb F_p$, adjoin a new vertex $w$ to $G_p^*$ and join $w$ to $\mathcal T(-2)$, $\mathcal T(\mathrm{i})$, and $\mathcal T(-\zeta)$ by single edges. The resulting multigraph is cubic. For every bridge of $G_p^*$, Condition~\ref{cond:p1-bridge} places two of the three degree-$2$ vertices on one side and the third on the other, so the edges incident with $w$ reconnect the two sides; each new edge also lies on a cycle because $G_p^*$ is connected. Thus the cubic multigraph is connected and bridgeless. By Sch\"onberger's strengthening of Petersen's theorem~\cite[Satz~II]{Schonberger}, every edge of a connected bridgeless cubic multigraph lies in a $1$-factor. Choosing a $1$-factor containing $\{w,\mathcal T(-\zeta)\}$ and deleting $w$ gives a $1$-factor of $G_p^*-\mathcal T(-\zeta)$; choosing one containing $\{w,\mathcal T(\mathrm{i})\}$ gives a $1$-factor of $G_p^*-\mathcal T(\mathrm{i})$.

Now suppose that $x^2+x-1=0$ has a root $\rho\in\mathbb F_p$, and set $G_p^\circ:=G_p^*-\mathcal T(\rho)$. Since $\mathcal T(\rho)$ is a leaf of $G_p^*$, the graph $G_p^\circ$ is connected and of even order. Its four degree-$2$ vertices are $\mathcal T(-\zeta)$, $\mathcal T(-2)$, $\mathcal T(\mathrm{i})$, and $\mathcal T(\rho-1)$, where $\mathcal T(\rho-1)$ loses one incident edge when $\mathcal T(\rho)$ is deleted. If a bridge of $G_p^\circ$ had an odd component disjoint from these four vertices, the same edge would separate an odd component in $G_p^*$ containing none of $\mathcal T(-2)$, $\mathcal T(\mathrm{i})$, $\mathcal T(-\zeta)$, and $\mathcal T(\rho)$, contradicting Condition~\ref{cond:p1-bridge}.

\begin{lemma}\label{lem:p1-core-factor}
Let $p>13$ be a prime with $p\equiv1\pmod{12}$ such that $x^2+x-1=0$ has a root $\rho\in\mathbb F_p$. If $G_p^*$ satisfies Condition~\ref{cond:p1-bridge}, then $G_p^\circ:=G_p^*-\mathcal T(\rho)$ has a $1$-factor.
\end{lemma}

\begin{proof}
Let $U\subseteq V(G_p^\circ)$ and set $q:=\operatorname{odd}(G_p^\circ-U)$. Let $r$ be the number of vertices of $U$ among $\mathcal T(-\zeta)$, $\mathcal T(-2)$, $\mathcal T(\mathrm{i})$, and $\mathcal T(\rho-1)$, and let $s$ be the number of these four vertices lying in odd components of $G_p^\circ-U$.

For the vertex set $L$ of an odd component of $G_p^\circ-U$, let $s_L$ be the number of these four vertices contained in $L$. The degree-sum formula gives $|\partial(L)|\equiv1-s_L\pmod2$. If $s_L=0$ and $|\partial(L)|=1$, then $\mathcal T(\rho-1)\notin L$, so the same edge separates $L$ in $G_p^*$, contradicting Condition~\ref{cond:p1-bridge}. Hence $|\partial(L)|\ge3-s_L$ in all cases. Since the boundary sets $\partial(L)$ of distinct odd components are pairwise disjoint and every such edge has its other endpoint in $U$, summing over all odd components gives
$$
3q-s\leq\sum_L|\partial(L)|\leq3|U|-r.
$$
If $q>|U|$, then $q\equiv|U|\pmod2$ because $|V(G_p^\circ)|$ is even, so $q\ge|U|+2$. This implies $6\le s-r$, whereas $s\le4-r$ gives $s-r\le4-2r\le4$, a contradiction. By Tutte's $1$-factor theorem~\cite{Tutte}, $G_p^\circ$ has a $1$-factor.
\end{proof}

\begin{lemma}\label{lem:p1-four-defect}
Let $G$ be a connected loopless multigraph of even order with a $1$-factor. Suppose that exactly four vertices of $G$ have degree $2$, every other vertex has degree $3$, and every odd component separated by a bridge contains at least one of the four degree-$2$ vertices. Fix a degree-$2$ vertex $u$. Then at most one of the other three degree-$2$ vertices $v$ has the property that $G-u-v$ has no $1$-factor.
\end{lemma}

\begin{proof}
Write $G_u:=G-u$. Deleting the edge incident with $u$ from a $1$-factor of $G$ leaves a matching of $G_u$ with one unmatched vertex, so every maximum matching of $G_u$ leaves exactly one vertex unmatched. Let $(D_{\rm GE},A_{\rm GE},C_{\rm GE})$ be the Gallai--Edmonds decomposition of $G_u$~\cite[Theorem~3.2.1]{LovaszPlummer}: $D_{\rm GE}$ is the set of vertices unmatched in some maximum matching of $G_u$, $A_{\rm GE}$ consists of the vertices outside $D_{\rm GE}$ having a neighbour in $D_{\rm GE}$, and $C_{\rm GE}:=V(G_u)\setminus(D_{\rm GE}\cup A_{\rm GE})$. Then $G_u[D_{\rm GE}]$ has $|A_{\rm GE}|+1$ odd components, there are no edges between $D_{\rm GE}$ and $C_{\rm GE}$, and a vertex $z\in V(G_u)$ belongs to $D_{\rm GE}$ if and only if $G_u-z$ has a $1$-factor.

Suppose toward a contradiction that $v$ and $w$ are distinct degree-$2$ vertices different from $u$ such that neither $G-u-v$ nor $G-u-w$ has a $1$-factor. Then $v,w\notin D_{\rm GE}$. Set $X_{\rm GE}:=A_{\rm GE}\cup\{u\}$ and $h:=|X_{\rm GE}|=|A_{\rm GE}|+1$. Let $r_D,r_A,r_C$ denote the numbers of the four degree-$2$ vertices lying in $D_{\rm GE},X_{\rm GE},C_{\rm GE}$, respectively, and let $q_C$ be the number of components of $G_u[C_{\rm GE}]$. With $a:=|\{v,w\}\cap A_{\rm GE}|$, we have $r_A\ge1+a$ and $r_C\ge2-a$.

We compare the number of boundary edges leaving the components of $G_u[D_{\rm GE}]$ and $G_u[C_{\rm GE}]$ with the total degree available in $X_{\rm GE}$. For each odd component $L$ of $G_u[D_{\rm GE}]$, let $r_L$ be the number of the four degree-$2$ vertices it contains. The degree-sum formula and the bridge hypothesis give $|\partial_G(L)|\ge3-r_L$. Hence the $h$ odd components of $G_u[D_{\rm GE}]$ have total boundary at least $3h-r_D$. Every component of $G_u[C_{\rm GE}]$ has positive boundary in $G$; since there are no edges between $D_{\rm GE}$ and $C_{\rm GE}$ in $G_u$, the boundary sets $\partial_G(L)$ of all components of $G_u[D_{\rm GE}]$ and $G_u[C_{\rm GE}]$ are pairwise disjoint and have their other endpoints in $X_{\rm GE}$. Since the sum of the degrees in $G$ of the vertices in $X_{\rm GE}$ is $3h-r_A$, we obtain
$$
3h-r_D+q_C\le3h-r_A.
$$
Using $r_D+r_A+r_C=4$ gives $q_C\le r_D-r_A=4-2r_A-r_C\le-a$. This is impossible if $a\ge1$; if $a=0$, then $v,w\in C_{\rm GE}$ forces $q_C\ge1$, whereas $q_C\le0$.
\end{proof}

Let $G$ be a multigraph with a $1$-factor $\mathcal M$. A path in $G$ is called \emph{$\mathcal M$-saturated} if its edges alternate between $\mathcal M$ and $E(G)\setminus\mathcal M$ and its first and last edges belong to $\mathcal M$.

\begin{lemma}\label{lem:p1-compatible-paths}
Let $p>13$ be a prime with $p\equiv1\pmod{12}$ such that $x^2+x-1=0$ has a root $\rho\in\mathbb F_p$, and suppose that $G_p^*$ satisfies Condition~\ref{cond:p1-bridge}. Let $\mathcal M$ be a $1$-factor of $G_p^\circ:=G_p^*-\mathcal T(\rho)$. If $G_p^\circ-\mathcal T(-\zeta)-\mathcal T(\mathrm{i})$ and $G_p^\circ-\mathcal T(\rho-1)-\mathcal T(\mathrm{i})$ have $1$-factors, while $G_p^\circ-\mathcal T(-\zeta)-\mathcal T(\rho-1)$ has none, then there exist $\mathcal M$-saturated paths in $G_p^\circ$ from $\mathcal T(\mathrm{i})$ to $\mathcal T(-\zeta)$ and from $\mathcal T(\mathrm{i})$ to $\mathcal T(\rho-1)$ that share an initial subpath from $\mathcal T(\mathrm{i})$ and have no other vertices in common.
\end{lemma}

\begin{proof}
Write $u:=\mathcal T(\mathrm{i})$, $b:=\mathcal T(-\zeta)$, and $c:=\mathcal T(\rho-1)$. Taking the symmetric differences of $\mathcal M$ with $1$-factors of $G_p^\circ-u-b$ and $G_p^\circ-u-c$ gives $\mathcal M$-saturated paths $Q_b$ from $u$ to $b$ and $Q_c$ from $u$ to $c$. Choose such a pair $(Q_b,Q_c)$ first minimizing $|E(Q_b)\cup E(Q_c)|$ and, subject to this, maximizing the length of their common initial subpath from $u$.

Every vertex shared by $Q_b$ and $Q_c$ lies on a common edge of $\mathcal M$, since its unique incident $\mathcal M$-edge belongs to both paths. Suppose that $Q_b$ and $Q_c$ share a vertex outside their common initial subpath. Let $e\in\mathcal M$ be the last common edge along $Q_b$, with endpoints $v$ and $w$ appearing in that order along $Q_b$. Then the subpath $Q_b[w,b]$ meets $Q_c$ only at $w$, as any further common vertex would lie on a common $\mathcal M$-edge after $e$.

If $Q_c$ traverses $e$ from $w$ to $v$, then following $Q_b$ backwards from $b$ to $w$ and then $Q_c$ from $w$ to $c$ gives a simple $\mathcal M$-saturated path between $b$ and $c$ (even when $w=b$, where the first part has length zero). Taking the symmetric difference of $\mathcal M$ with this path yields a $1$-factor of $G_p^\circ-b-c$, a contradiction.

Otherwise, $Q_c$ traverses $e$ from $v$ to $w$. Replacing $Q_b[u,w]$ by $Q_c[u,w]$ while retaining $Q_b[w,b]$ gives a simple $\mathcal M$-saturated path $Q_b'$ from $u$ to $b$: indeed, $Q_c[u,w]$ ends with $e\in\mathcal M$, and $Q_b[w,b]$, when nonempty, starts with an edge outside $\mathcal M$. Moreover, $E(Q_b')\cup E(Q_c)\subseteq E(Q_b)\cup E(Q_c)$. Strict containment contradicts the minimality of $|E(Q_b)\cup E(Q_c)|$, whereas equality contradicts the maximality of the common initial subpath, which now extends at least to $w$.
\end{proof}

To summarize, when $p>13$, $G_p^*$ satisfies Condition~\ref{cond:p1-bridge}, and $x^2+x-1=0$ has a root $\rho\in\mathbb F_p$, two cases arise for $G_p^\circ=G_p^*-\mathcal T(\rho)$:
\begin{enumerate}[label=(\arabic*)]
\item If $G_p^\circ-\mathcal T(-\zeta)-\mathcal T(\rho-1)$ has a $1$-factor, then adjoining the pendant edge $\{\mathcal T(\rho-1),\mathcal T(\rho)\}$ gives a $1$-factor of $G_p^*-\mathcal T(-\zeta)$, while Lemma~\ref{lem:p1-core-factor} provides a $1$-factor of $G_p^*-\mathcal T(\rho)$.
\item If $G_p^\circ-\mathcal T(-\zeta)-\mathcal T(\rho-1)$ has no $1$-factor, then applying Lemma~\ref{lem:p1-four-defect} twice (with $u=\mathcal T(-\zeta)$ and with $u=\mathcal T(\rho-1)$) shows that both $G_p^\circ-\mathcal T(-\zeta)-\mathcal T(\mathrm{i})$ and $G_p^\circ-\mathcal T(\rho-1)-\mathcal T(\mathrm{i})$ have $1$-factors; Lemma~\ref{lem:p1-compatible-paths} then provides compatible $\mathcal M$-saturated paths from $\mathcal T(\mathrm{i})$ to $\mathcal T(-\zeta)$ and to $\mathcal T(\rho-1)$ for any $1$-factor $\mathcal M$ of $G_p^\circ$.
\end{enumerate}

\subsection{Lifting to the full K\"ohler graph}\label{subsec:p1-lift}

Throughout this subsection, we write the cyclic group $M_p=\mathbb F_p^\times/\{\pm1\}$ of order $(p-1)/2$ additively, with identity element $0$. Since $p\equiv1\pmod{12}$, the group $M_p$ has a unique subgroup $W$ of order $3$ and a unique element $\tau$ of order $2$.

For each $a\in X_p$, let $\mathcal V(a):=\{\mathcal T_{\pm}(T):T\in\mathcal T(a)\}$ be the set of vertices of $\mathcal K_p$ belonging to the affine triangle orbit $\mathcal T(a)$. When $a\in X_p^*$, the action of $M_p$ on $\mathcal V(a)$ is regular, and we index $\mathcal V(a)=\{\mathcal T(a)_x:x\in M_p\}$ so that the action of $y\in M_p$ sends $\mathcal T(a)_x$ to $\mathcal T(a)_{x+y}$. For the exceptional orbit $\mathcal T(\zeta)$, each vertex in $\mathcal V(\zeta)$ has stabilizer $W$ in $M_p$, so we index $\mathcal V(\zeta)=\{\mathcal T(\zeta)_{x+W}:x+W\in M_p/W\}$ by the cosets of $W$ in $M_p$, with the action of $y\in M_p$ adding $y$ to the coset index.

\begin{lemma}\label{lem:p1-scalar-stabilizer}
Let $p>13$ be a prime with $p\equiv1\pmod{12}$, and let $e\in E(\mathcal K_p)$. If a nonidentity element of $M_p$ fixes $e$, then the affine trapezoid orbit containing $e$ is the orbit omitted as a loop at $\mathcal T(\mathrm{i})$, and the stabilizer of $e$ in $M_p$ is $\{0,\tau\}$.
\end{lemma}

\begin{proof}
Translating a representative trapezoid $P(a)=\{0,1,a,1+a\}$ by $-(1+a)/2$ centers its reflection symmetry at $0$, so we may write it as $P=\{\pm u,\pm v\}$. A scalar $\mu\in\mathbb F_p^\times$ whose class in $M_p$ fixes $\mathcal P_{\pm}(P)$ satisfies $\mu P=P+t$ for some $t\in\mathbb F_p$. Summing the four elements gives $4t=0$, so $t=0$ and $\mu P=P$. If $\mu\notin\{\pm1\}$, multiplication by $\mu$ cannot preserve both $\{u,-u\}$ and $\{v,-v\}$ separately. After replacing $v$ by $-v$ if necessary, we have $\mu u=v$ and $\mu v\in\{u,-u\}$. Since $\mu^2\neq1$, we obtain $\mu v=-u$ and $\mu^2=-1$. Thus $\mu=\pm\mathrm{i}$, whose class in $M_p$ is $\tau$, and $P$ belongs to the affine trapezoid orbit omitted as a loop at $\mathcal T(\mathrm{i})$.
\end{proof}

\begin{lemma}\label{lem:p1-ordinary-lift}
Let $p>13$ be a prime with $p\equiv1\pmod{12}$. Let $a,b\in X_p^*$, and let $e$ be an oriented edge of $G_p^*$ from $\mathcal T(a)$ to $\mathcal T(b)$. Then there exists $c_e\in M_p$ such that the edges of $\mathcal K_p$ belonging to the affine trapezoid orbit represented by $e$ are precisely
$$
\bigl\{\{\mathcal T(a)_x,\mathcal T(b)_{x+c_e}\}:x\in M_p\bigr\}.
$$
In particular, these edges define a bijection from $\mathcal V(a)$ to $\mathcal V(b)$.
\end{lemma}

\begin{proof}
Choose an edge of $\mathcal K_p$ in the affine orbit represented by $e$ with one endpoint $\mathcal T(a)_0$, and write its other endpoint as $\mathcal T(b)_{c_e}$. Developing this edge under $M_p$ yields the displayed family. Since $e$ is an edge of the loopless multigraph $G_p^*$, Lemma~\ref{lem:p1-scalar-stabilizer} shows that its stabilizer in $M_p$ is trivial, so the $|M_p|$ displayed edges are distinct and exhaust the $M_p$-orbit.
\end{proof}

For each oriented edge $e$ of $G_p^*$ from $\mathcal T(a)$ to $\mathcal T(b)$, let $\phi_e:M_p\to M_p$ be the translation $\phi_e(x):=x+c_e$, where $c_e\in M_p$ is given by Lemma~\ref{lem:p1-ordinary-lift}; reversing the orientation of $e$ replaces $\phi_e$ by $\phi_e^{-1}$. Since $\phi_e$ is a translation of $M_p$, for every $S\subseteq M_p$ and $t\in M_p$ we have
$$
\phi_e(M_p\setminus S)=M_p\setminus\phi_e(S),
\qquad
\phi_e(S+t)=\phi_e(S)+t.
$$
Consequently, transport along any path in $G_p^*$ acts on subsets of $M_p$ by translation: it preserves complements, coset transversals of $W$, the condition $S\cap(S+\tau)=\varnothing$, and $1$-factors in the Cayley graph on $M_p$ generated by any fixed $\delta\in M_p$.

\begin{lemma}\label{lem:p1-branch-lift}
Let $p>13$ be a prime with $p\equiv1\pmod{12}$. After shifting the index origin of $\mathcal V(-\zeta)$ by a suitable element of $M_p$, the edges of $\mathcal K_p$ belonging to the affine trapezoid orbit joining $\mathcal T(-\zeta)$ and $\mathcal T(\zeta)$ are precisely
$$
\bigl\{\{\mathcal T(\zeta)_{x+W},\mathcal T(-\zeta)_x\}:x\in M_p\bigr\}.
$$
In particular, each vertex in $\mathcal V(\zeta)$ is incident with exactly three of these edges.
\end{lemma}

\begin{proof}
Choose an edge in this affine trapezoid orbit incident with $\mathcal T(\zeta)_W$. By shifting the indexing of $\mathcal V(-\zeta)$, we may assume that its other endpoint is $\mathcal T(-\zeta)_0$. Since $\mathcal T(-\zeta)\neq\mathcal T(\zeta)$, Lemma~\ref{lem:p1-scalar-stabilizer} shows that the stabilizer of this edge in $M_p$ is trivial. Developing the edge under $M_p$ yields the displayed family, and each coset $x+W$ contains three elements of $M_p$.
\end{proof}

From now on, we fix the indexing of $\mathcal V(-\zeta)$ so that Lemma~\ref{lem:p1-branch-lift} holds, and compute all edge translations $\phi_e$ with respect to this choice.

\begin{lemma}\label{lem:p1-square-lift}
Let $p>13$ be a prime with $p\equiv1\pmod{12}$. The affine trapezoid orbit omitted as a loop at $\mathcal T(\mathrm{i})$ forms a $1$-factor of $\mathcal K_p$ on $\mathcal V(\mathrm{i})$, whose edges pair $\mathcal T(\mathrm{i})_x$ with $\mathcal T(\mathrm{i})_{x+\tau}$ for $x\in M_p$.
\end{lemma}

\begin{proof}
Multiplication by $\mathrm{i}$ sends $P(\mathrm{i})$ to $P(\mathrm{i})-1$, so its class $\tau\in M_p$ fixes the corresponding edge of $\mathcal K_p$. By Lemma~\ref{lem:p1-scalar-stabilizer}, the edge stabilizer in $M_p$ is $\{0,\tau\}$. Since $M_p$ acts freely on $\mathcal V(\mathrm{i})$, the involution $\tau$ exchanges the two endpoints of this edge, so developing it under $M_p$ pairs $\mathcal T(\mathrm{i})_x$ with $\mathcal T(\mathrm{i})_{x+\tau}$.
\end{proof}

\begin{lemma}\label{lem:p1-golden-lift}
Let $p>13$ be a prime with $p\equiv1\pmod{12}$ such that $x^2+x-1=0$ has a root $\rho\in\mathbb F_p$. Then the edges of $\mathcal K_p$ belonging to the affine trapezoid orbit omitted as a loop at $\mathcal T(\rho)$ are
$$
\bigl\{\{\mathcal T(\rho)_x,\mathcal T(\rho)_{x+\delta_\rho}\}:x\in M_p\bigr\}
$$
for some $\delta_\rho\in M_p$ satisfying $\delta_\rho\neq0$ and $2\delta_\rho\neq0$.
\end{lemma}

\begin{proof}
Take $\mathcal T_{\pm}(\{0,1,\rho\})$ to have index $0$ in $\mathcal V(\rho)$. The edge represented by $P(\rho)$ joins it to $\mathcal T_{\pm}(\{0,1,1+\rho\})$. Since $\rho(1+\rho)=1$, multiplication by $\rho$ sends $\{0,1,1+\rho\}$ to $\{0,\rho,1\}$, so the other endpoint has index $\delta_\rho:=[\rho^{-1}]\in M_p$. Shifting the index origin of $\mathcal V(\rho)$ does not alter this difference, and developing the edge under $M_p$ gives the displayed family.

If $\delta_\rho=0$, then $\rho=\pm1$, which is incompatible with $\rho^2+\rho-1=0$. If $2\delta_\rho=0$, then $\rho^2=\pm1$; here $\rho^2=1$ gives $\rho=0$, while $\rho^2=-1$ gives $\rho=2$ and $p=5$, contradicting $p>13$.
\end{proof}

By Lemma~\ref{lem:p1-branch-lift}, any subset $R\subseteq M_p$ containing one representative from each coset of $W$ selects a matching between $\{\mathcal T(-\zeta)_x:x\in R\}$ and all of $\mathcal V(\zeta)$.

\begin{lemma}\label{lem:p1-transversal}
Let $p$ be a prime with $p\equiv1\pmod{12}$, and let $\delta\in M_p\setminus\{0\}$. There exists a coset transversal $R\subseteq M_p$ of $W$ such that the graph on $M_p\setminus R$ with edges $\{x,x+\delta\}$ has a $1$-factor.
\end{lemma}

\begin{proof}
The connected components of the Cayley graph on $M_p$ with edges $x\sim x+\delta$ are the cosets of $K:=\langle\delta\rangle$. If $\operatorname{ord}(\delta)=2$, each component is a single edge, and translation by $W$ permutes these components in triples; choosing one component from each triple gives a transversal $R$ of $W$ whose complement is a union of disjoint edges. Assume henceforth that $q:=\operatorname{ord}(\delta)\ge3$, so each component is a cycle of length $q$.

If $W\le K$, then $3\mid q$. On each cycle, label the vertices consecutively along $\delta$; choosing the first $q/3$ consecutive vertices gives one representative from each coset of $W$ in that cycle, and leaves a path on $2q/3$ vertices, which has a $1$-factor since $2q/3$ is even.

If $W\cap K=\{0\}$, the cycles occur in triples permuted by $W$. Label the three cycles in a triple by rows $0,1,2$ with origins in the same coset of $W$, so each column is a coset of $W$. If $q$ is even, put all of row $0$ into $R$; the remaining two rows are even cycles. If $q$ is odd, choose row $0$ in the first $q-2$ columns, row $1$ in column $q-2$, and row $2$ in column $q-1$; the unselected vertices form one adjacent pair in row $0$ and a path on $q-1$ vertices in each of rows $1$ and $2$, each of which has a $1$-factor since $q-1$ is even.
\end{proof}

\begin{lemma}\label{lem:p1-path-lift}
Let $p>13$ be a prime with $p\equiv1\pmod{12}$, and let $U=\mathcal T(\mathrm{i})$, or $U=\mathcal T(\rho)$ when $x^2+x-1=0$ has a root $\rho\in\mathbb F_p$. If both $G_p^*-\mathcal T(-\zeta)$ and $G_p^*-U$ have $1$-factors, then $\mathcal K_p$ has a $1$-factor.
\end{lemma}

\begin{proof}
Let $\mathcal M_{-\zeta}$ and $\mathcal M_U$ be $1$-factors of $G_p^*-\mathcal T(-\zeta)$ and $G_p^*-U$, respectively. Their symmetric difference contains a unique alternating path $v_0e_1v_1\cdots e_rv_r$ from $v_0=\mathcal T(-\zeta)$ to $v_r=U$, where $e_1\in\mathcal M_U$ and $e_r\in\mathcal M_{-\zeta}$, so $r$ is even.

Set $\delta:=\tau$ if $U=\mathcal T(\mathrm{i})$, and $\delta:=\delta_\rho$ (from Lemma~\ref{lem:p1-golden-lift}) if $U=\mathcal T(\rho)$. Choose a transversal $R_r\subseteq M_p$ of $W$ by Lemma~\ref{lem:p1-transversal} for this $\delta$. Orient each $e_j$ from $v_{j-1}$ to $v_j$, and define $R_{j-1}:=\phi_{e_j}^{-1}(R_j)$ recursively for $j=r,r-1,\ldots,1$. From the affine trapezoid orbit represented by $e_j$, select the edges of $\mathcal K_p$ with initial indices at $v_{j-1}$ in $M_p\setminus R_{j-1}$ when $j$ is odd, and in $R_{j-1}$ when $j$ is even. At each internal vertex $v_j$ ($1\le j<r$), the two incident path edges cover the indices in $R_j$ and $M_p\setminus R_j$, respectively.

Since $r$ is even, $e_r$ covers the vertices represented by $U$ with indices in $R_r$, while Lemma~\ref{lem:p1-transversal} matches the remaining indices $M_p\setminus R_r$ via pairs $\{x,x+\delta\}$, which correspond to edges of $\mathcal K_p$ by Lemma~\ref{lem:p1-square-lift} or~\ref{lem:p1-golden-lift}. At $v_0=\mathcal T(-\zeta)$, the edge $e_1$ leaves the indices in $R_0$ uncovered; since $R_0$ is a translate of $R_r$, it is a transversal of $W$, so the edges $\{\mathcal T(\zeta)_{x+W},\mathcal T(-\zeta)_x\}$ ($x\in R_0$) from Lemma~\ref{lem:p1-branch-lift} cover all of $\mathcal V(\zeta)$ and the remaining vertices of $\mathcal V(-\zeta)$. Finally, every vertex on the path other than $v_0$ has its unique $\mathcal M_{-\zeta}$-edge on the path, so every edge of $\mathcal M_{-\zeta}$ outside the path has both endpoints outside the path; taking all corresponding edges of $\mathcal K_p$ covers all remaining vertices.
\end{proof}

When $p=13$, the four trapezoid base blocks $\{0,1,3,4\}$, $\{0,2,5,7\}$, $\{0,2,6,9\}$, and $\{0,1,5,6\}$ yield a $1$-factor of $\mathcal K_{13}$: the four corresponding edges pair up the eight vertices of $\mathcal K_{13}$.

For $p>13$ satisfying Condition~\ref{cond:p1-bridge}, Lemma~\ref{lem:p1-path-lift} gives a $1$-factor of $\mathcal K_p$ when $x^2+x-1=0$ has no solution in $\mathbb F_p$. Suppose now that this equation has a root $\rho\in\mathbb F_p$. If $G_p^\circ-\mathcal T(-\zeta)-\mathcal T(\rho-1)$ has a $1$-factor, the same lemma applies with $U=\mathcal T(\rho)$, using the $1$-factors obtained in Subsection~\ref{subsec:p1-tutte}.

It remains to treat the case in which $G_p^\circ-\mathcal T(-\zeta)-\mathcal T(\rho-1)$ has no $1$-factor. The two compatible saturated paths supplied by Lemma~\ref{lem:p1-compatible-paths} will be lifted simultaneously. At their branch vertex, the two branches will use the index sets $R$ and $R+\tau$, so these sets must be disjoint.

The set $R$ must contain one representative from each coset of $W$, allowing the uncovered vertices represented by $\mathcal T(-\zeta)$ to be matched with the exceptional orbit. The union $R\cup(R+\tau)$ is invariant under addition of $\tau$, allowing the uncovered vertices represented by $\mathcal T(\mathrm{i})$ to be matched by the omitted loop orbit. Finally, the complement of $R$ must admit a matching using the omitted loop orbit at $\mathcal T(\rho)$. The following lemma supplies a transversal with these properties and will be applied with $\delta=\delta_\rho$.

\begin{lemma}\label{lem:p1-separated-transversal}
Let $p$ be a prime with $p\equiv1\pmod{12}$, and let $\delta\in M_p$ satisfy $\delta\neq0$ and $2\delta\neq0$. There exists a coset transversal $R\subseteq M_p$ of $W$ such that $R\cap(R+\tau)=\varnothing$ and the graph on $M_p\setminus R$ with edges $\{x,x+\delta\}$ has a $1$-factor.
\end{lemma}

\begin{proof}
Set $K:=\langle\delta\rangle$ and $q:=|K|\ge3$. Each coset of $K$ is a cycle of length $q$ under $x\sim x+\delta$, and either $W\le K$ or $W\cap K=\{0\}$.

Suppose first that $W\le K$, so $3\mid q$ and $W=\langle(q/3)\delta\rangle$. If $q$ is even, then $\tau=(q/2)\delta$; in each coset $x+K$, selecting $\{x+t\delta:0\le t<q/3\}$ gives a transversal of the $W$-cosets in $x+K$ that is disjoint from its translate by $\tau$, and leaves a path on $2q/3$ vertices, which we match consecutively. If $q$ is odd, then $\tau\notin K$ pairs the cosets of $K$ as $x+K$ and $x+\tau+K$; setting $d:=q/3$ and selecting $\{x+t\delta:0\le t<d\}$ in the first cycle and $\{x+\tau+t\delta:d\le t<2d\}$ in the second gives a transversal of $W$ disjoint from its $\tau$-translate and leaves a path on $2d$ vertices in each cycle, which we match consecutively.

Now suppose that $W\cap K=\{0\}$, and fix a generator $\omega$ of $W$. On each coset $x+(W+K)$, use the coordinates
$$
v_{r,t}:=x+r\omega+t\delta,\qquad r\in\mathbb Z_3,\quad t\in\mathbb Z_q.
$$
For each fixed $r$ the points form a row, which is a $K$-cycle; for each fixed $t$ they form a column, which is a $W$-coset. Thus we must select exactly one point in each column.

If $q=2h$ is even, then $\tau=h\delta$, the unique element of order $2$ in $M_p$, and $v_{r,t}+\tau=v_{r,t+h}$. Partition the columns into adjacent pairs $\{2j,2j+1\}$ ($j\in\mathbb Z_h$). Choose a row $\chi(j)\in\mathbb Z_3$ for each pair, and define
$$
R_x:=\bigcup_{j\in\mathbb Z_h}\{v_{\chi(j),2j},v_{\chi(j),2j+1}\}.
$$
This selects one point per column, and its complement in $x+(W+K)$ has the $1$-factor
$$
\bigl\{\{v_{r,2j},v_{r,2j+1}\}:j\in\mathbb Z_h,\ r\in\mathbb Z_3\setminus\{\chi(j)\}\bigr\},
$$
since these pairs partition the unselected points and each pair is an edge of $x\sim x+\delta$.

If $h$ is even, translation by $\tau$ sends pair $j$ to pair $j+h/2$ in the same row; choosing $\chi(j)=0$ for $0\le j<h/2$ and $\chi(j)=1$ for $h/2\le j<h$ ensures that translated pairs use different rows, so $R_x\cap(R_x+\tau)=\varnothing$.

If $h=2a+1$ is odd, then $h\ge3$ because $q\neq2$. Translation by $\tau$ acts as
$$
v_{r,2j}+\tau=v_{r,2(j+a)+1},\qquad
v_{r,2j+1}+\tau=v_{r,2(j-a)},
$$
with column indices modulo $2h$ and pair indices modulo $h$; the second identity uses $a+1\equiv-a\pmod h$. Thus the two translated points from pair $j$ lie in pairs $j+a$ and $j-a$, respectively, and remain in row $\chi(j)$. Consequently,
$$
R_x\cap(R_x+\tau)=\varnothing
\quad\Longleftrightarrow\quad
\chi(j)\ne\chi(j+a),\ \chi(j)\ne\chi(j-a)\quad\text{for every }j\in\mathbb Z_h.
$$
These inequalities say exactly that $\chi$ properly colours the graph on $\mathbb Z_h$ with edges $j\sim j+a$. Since $\gcd(h,a)=1$, this graph is the single cycle $0,a,\ldots,(h-1)a$. Setting $\chi(ka)=k\bmod2$ for $0\le k\le h-2$ and $\chi((h-1)a)=2$, the colours alternate until the final vertex, whose colour differs from both neighbours. Thus $R_x\cap(R_x+\tau)=\varnothing$. Since translation by $\tau\in K$ preserves each coset $x+(W+K)$, taking the union of $R_x$ over all cosets gives $R$.

Finally, suppose that $q$ is odd. Then $|W+K|=3q$ is odd, so $\tau\notin W+K$ and translation by $\tau$ pairs its cosets. Use the coordinates $v_{r,t}=x+r\omega+t\delta$ in the first coset and $v'_{r,t}:=v_{r,t}+\tau$ in the second, so that translation by $\tau$ exchanges corresponding points. In columns $0,\ldots,q-1$, choose rows $(0,\ldots,0,1,2)$ in the first coset and $(1,\ldots,1,2,0)$ in the second. Each column has one selection, and corresponding columns use different rows, so no selected point is the $\tau$-translate of another. In the first coset, the unselected points in row $0$ form the adjacent pair in columns $q-2,q-1$, while each of rows $1$ and $2$ has one selected point and leaves a path on $q-1$ vertices, which we match consecutively. The second coset has the same matching with the rows cyclically permuted. Taking these selections and matchings over all paired cosets completes the proof.
\end{proof}

\begin{lemma}\label{lem:p1-multipath-lift}
Let $p>13$ be a prime with $p\equiv1\pmod{12}$ such that $x^2+x-1=0$ has a root $\rho\in\mathbb F_p$. Write $u:=\mathcal T(\mathrm{i})$, $b:=\mathcal T(-\zeta)$, $c:=\mathcal T(\rho-1)$, $r:=\mathcal T(\rho)$, and $G_p^\circ:=G_p^*-r$. Let $\mathcal M$ be a $1$-factor of $G_p^\circ$. If $G_p^\circ$ contains $\mathcal M$-saturated paths from $u$ to $b$ and from $u$ to $c$ that share an initial subpath from $u$ and have no other vertices in common, then $\mathcal K_p$ has a $1$-factor.
\end{lemma}

\begin{proof}
Let $v_*$ be the last vertex of the common initial subpath. Denote the common path from $u$ to $v_*$ by $Q_0$, the remaining subpath to $b$ by $Q_b$, and the remaining subpath to $c$, followed by the pendant edge $\{c,r\}$ of $G_p^*$, by $Q_\rho$. Each of $c$ and $r$ corresponds to $|M_p|$ vertices of $\mathcal K_p$, so Lemma~\ref{lem:p1-ordinary-lift} applies to $\{c,r\}$, which lies outside $\mathcal M$ since $r\notin V(G_p^\circ)$.

Both original $\mathcal M$-saturated paths start with the unique $\mathcal M$-edge at $u$, so $Q_0$ has positive length. Its last edge also belongs to $\mathcal M$: if neither path ends at $v_*$ and the last common edge were outside $\mathcal M$, both paths would continue along the same $\mathcal M$-edge incident with $v_*$. Thus every nonempty remaining path among $Q_b$ and $Q_\rho$ starts with an edge outside $\mathcal M$. The last edge of $Q_b$, when present, belongs to $\mathcal M$, while $Q_\rho$ ends with $\{c,r\}\notin\mathcal M$. We allow $v_*=b$, when $Q_b$ has length zero, and $v_*=c$, when $Q_\rho$ consists only of $\{c,r\}$.

Orient $Q_0$ from $u$ to $v_*$ and orient $Q_b$ and $Q_\rho$ away from $v_*$. Choose $R_*\subseteq M_p$ by Lemma~\ref{lem:p1-separated-transversal} with $\delta=\delta_\rho$ from Lemma~\ref{lem:p1-golden-lift}, and set $\widehat R_*:=R_*\cup(R_*+\tau)$. This union is disjoint, and $\widehat R_*+\tau=\widehat R_*$. Since $R_*\cap(R_*+\tau)=\varnothing$, the three sets
$$
M_p\setminus\widehat R_*,\qquad R_*+\tau,\qquad R_*
$$
form a partition of $M_p$.

Transport these sets along the three paths separately by taking successive inverse images of $\widehat R_*$ along $Q_0$, successive images of $R_*+\tau$ along $Q_b$, and successive images of $R_*$ along $Q_\rho$ under the edge translations $\phi_e$; denote the final image at $r$ along $Q_\rho$ by $R_\rho$. For an oriented edge $e$ from $v$ to $w$ on one of these three paths, let $S\subseteq M_p$ be the set carried by that path at $v$, so that $\phi_e(S)$ is the set carried by that path at $w$. If $e\notin\mathcal M$, select the edges of $\mathcal K_p$ in the affine orbit represented by $e$ whose indices at $v$ lie in $S$; if $e\in\mathcal M$, select those whose indices at $v$ lie in $M_p\setminus S$. Since $\phi_e$ preserves complements, Lemma~\ref{lem:p1-ordinary-lift} shows that the selected edges cover the indices in $\phi_e(S)$ (when $e\notin\mathcal M$) or $M_p\setminus\phi_e(S)$ (when $e\in\mathcal M$) at $w$.

At an internal vertex of any of the three paths, other than $v_*$, one incident path edge belongs to $\mathcal M$ and the other does not. They therefore cover the subset carried by that path and its complement, once each. This includes $c$ when $c\ne v_*$, since its incoming edge belongs to $\mathcal M$ and its outgoing edge is $\{c,r\}\notin\mathcal M$. If $v_*\ne b$, the edges on $Q_0$, $Q_b$, and $Q_\rho$ incident with $v_*$ cover $M_p\setminus\widehat R_*$, $R_*+\tau$, and $R_*$, respectively, so every index at $v_*$ is covered exactly once. If $v_*=b$, there is no edge of $Q_b$, and the other two edges leave precisely $R_*+\tau$ uncovered at $b$. The case $v_*=c$ is included in the first case, with $Q_\rho$ consisting of $\{c,r\}$.

At $u$, the first path edge belongs to $\mathcal M$, so the uncovered indices form the assigned translate of $\widehat R_*$. This set is invariant under addition of $\tau$ and hence is partitioned into pairs $\{x,x+\tau\}$. The corresponding edges from Lemma~\ref{lem:p1-square-lift} cover exactly these vertices.

At $b$, the uncovered indices form the assigned translate of $R_*+\tau$, either because the last edge of $Q_b$ belongs to $\mathcal M$ or by the preceding check when $v_*=b$. This set contains one representative from each coset of $W$. For every uncovered index $x$, choose the edge joining $\mathcal T(-\zeta)_x$ to $\mathcal T(\zeta)_{x+W}$ in Lemma~\ref{lem:p1-branch-lift}. These edges cover each remaining vertex in $\mathcal V(-\zeta)$ and every vertex in $\mathcal V(\zeta)$ exactly once.

At $r$, the appended edge $\{c,r\}$ covers $R_\rho$. Since $R_\rho$ is a translate of $R_*$, its complement $M_p\setminus R_\rho$ has a $1$-factor using edges $x\sim x+\delta_\rho$ by Lemma~\ref{lem:p1-separated-transversal}, which correspond to edges of the affine trapezoid orbit omitted as a loop at $\mathcal T(\rho)$ in Lemma~\ref{lem:p1-golden-lift}.

Finally, every vertex on either original $\mathcal M$-saturated path has its unique $\mathcal M$-edge on that path. Thus every edge of $\mathcal M$ outside the union of the two paths has both endpoints outside that union. For each such edge, take all corresponding edges of $\mathcal K_p$. Together with the edges selected above, they form a $1$-factor of $\mathcal K_p$.
\end{proof}

Combining Lemmas~\ref{lem:p1-core-factor}--\ref{lem:p1-compatible-paths} with Lemmas~\ref{lem:p1-path-lift} and~\ref{lem:p1-multipath-lift}, we conclude that $\mathcal K_p$ has a $1$-factor for every prime $p\equiv1\pmod{12}$ provided that $G_p^*$ satisfies Condition~\ref{cond:p1-bridge} for $p>13$.

\subsection{Bridge condition for \texorpdfstring{$G_p^*$}{G(p)*}}
\label{subsec:p1-bridge}

Suppose toward a contradiction that $p>13$ and Condition~\ref{cond:p1-bridge} fails for a bridge $e$ of $G_p^*$. Let $L$ be the vertex set of the odd component of $G_p^*-e$. Then $L$ contains none of $\mathcal T(-2)$, $\mathcal T(\mathrm{i})$, and $\mathcal T(-\zeta)$, and if $x^2+x-1=0$ has a root $\rho\in\mathbb F_p$ with $\mathcal T(\rho)\in L$, failure of the condition implies $L\neq\{\mathcal T(\rho)\}$ and hence $\mathcal T(\rho-1)\in L$.

Write $Y:=\bigcup_{\mathcal T(a)\in L}H(a)$ and $\ell:=3|L|\equiv3\pmod6$. Since every vertex of $G_p^*$ corresponds to a six-element set $H(a)\subseteq X_p^*$, we have $|Y|=6|L|=2\ell$, and $Y$ is invariant under $\sigma(a)=a^{-1}$ and $\theta(a)=1-a$. The bridge $e$ contributes one inverse pair in $Y$ whose negative lies outside $Y$, while every other inverse pair at a vertex of $L$ has its negative in $Y$; if $x^2+x-1=0$ has a root $\rho\in\mathbb F_p$ with $\mathcal T(\rho)\in L$, the two inverse pairs belonging to the affine trapezoid orbit omitted as a loop at $\mathcal T(\rho)$ are negatives of one another and contribute nothing to the symmetric difference. Hence $|Y\mathbin{\triangle}(-Y)|=4$.

Since $\mathcal T(-2),\mathcal T(\mathrm{i}),\mathcal T(-\zeta)\notin L$, we have $Y\cup(-Y)\subseteq X_p^*$ and $\pm\mathrm{i}\notin Y\cup(-Y)$. By Lemma~\ref{lem:two-interval-form}, there exists $\xi\in\mathbb F_p$ such that
$$
Y=\{\xi+j:1\le j\le\ell\}\cup\{-\xi-j:0\le j<\ell\},\qquad \xi(\xi+\ell)=-1.
$$
Define $I_\xi:=\{\xi,\xi+1,\ldots,\xi+\ell\}$. Then $I_\xi\cap(-I_\xi)=\varnothing$ and $Y\cup(-Y)=I_\xi\cup(-I_\xi)$.

For every root $\rho\in\mathbb F_p$ of $x^2+x-1=0$, both roots and their negatives lie in $H(\rho)$. Since $Y$ is a union of sets $H(a)$, it contains a root if and only if $Y\cup(-Y)$ does. If $Y$ contains no root, Lemma~\ref{lem:bridge-congruences} gives
$$
p\mid\Lambda_\ell-\ell F_\ell,
\qquad
p\mid\Pi_\ell^2-1,
$$
where $F_j$, $\Lambda_j$, and $\Pi_\ell$ are defined as in Section~\ref{sec:p5}.

Otherwise, choose a root $\rho\in Y$ of $x^2+x-1=0$. Let $\varphi(z)=1/(z+1)$ act on $\mathbb P^1(\mathbb F_p)$. Its two fixed points belong to $H(\rho)\subseteq Y$. Consider the directed graph on $Y\cup(-Y)$ with edges $z\to\varphi(z)$ for $z\in-Y$. The fixed points give two $1$-cycles. By the argument in Lemma~\ref{lem:bridge-congruences}, all other directed cycles have length $\operatorname{ord}(\varphi)$, the order of $\varphi$ in $\operatorname{PGL}_2(\mathbb F_p)$, and occur in pairs under $z\mapsto-1/z$. Let $t$ be the common length of the two directed paths. Counting the $2\ell$ directed edges now gives
\begin{align*}
2\ell&=2t+2k\operatorname{ord}(\varphi)+2,\\
\ell-1&=t+k\operatorname{ord}(\varphi)
\end{align*}
for some $k\ge0$. Hence $\varphi^{\ell-1}(\xi)=\varphi^t(\xi)=-\xi$, which gives
$$
p\mid\Lambda_{\ell-1}-\ell F_{\ell-1},
\qquad
p\mid\Pi_\ell^2-1.
$$
The second congruence follows from the same product argument as in Lemma~\ref{lem:bridge-congruences}.

Let $d:=(p-7)/6-|L|$, the number of vertices in the other component of $G_p^*-e$. Since both $|V(G_p^*)|$ and $|L|$ are odd, $d$ is a positive even integer. Since $p=2\ell+6d+7$ and $|Y|=2\ell$, we have $p+1-|Y|=6d+8$. Lemma~\ref{lem:spectral-cutoff} applied to $Z=Y$ gives $\min\{6|L|,6d+8\}\le4000$. Thus either $|L|\le665$ or $d\le664$.

If $d\le664$, choose between $I_\xi$ and $-I_\xi$ as in Section~\ref{sec:p5}, and let $u$ be the least element of $I_\xi\cup(-I_\xi)$ among the integer representatives $0,1,\ldots,p-1$. Then $2\le u\le3d+3$, and the congruence $u(u+\ell)\equiv-1\pmod p$ gives $p\mid N_{d,u}$, where
$$
N_{d,u}:=(6d+7)u-2u^2-2
$$
satisfies $0<N_{d,u}\le1\,991\,008$.

\begin{lemma}\label{lem:p1-certificate-small}
For every $\ell\in\{3,9,\ldots,1995\}$, the integer $\gcd(\Pi_\ell^2-1,|\Lambda_\ell-\ell F_\ell|)$ has no prime divisor $p\equiv1\pmod{12}$ with $p\ge2\ell+13$.
\end{lemma}

\begin{proof}
Complete factorization for each of the $333$ values $\ell\in\{3,9,\ldots,1995\}$ yields no prime factor satisfying the stated conditions; see~\cite{LiuFengZenodo}.
\end{proof}

\begin{lemma}\label{lem:p1-certificate-golden}
For every $\ell\in\{9,15,\ldots,1995\}$, the integer $\gcd(\Pi_\ell^2-1,|\Lambda_{\ell-1}-\ell F_{\ell-1}|)$ has no prime divisor $p\equiv1\pmod{12}$ for which $(p-2\ell-7)/6$ is a positive even integer.
\end{lemma}

\begin{proof}
Complete factorization for $\ell\in\{9,15,\ldots,1995\}$ yields no prime factor satisfying the stated conditions (the omitted value $\ell=3$ corresponds to the allowed singleton component $\{\mathcal T(\rho)\}$); see~\cite{LiuFengZenodo}.
\end{proof}

\begin{lemma}\label{lem:p1-certificate-complement}
Let $d$ and $u$ be integers with $d$ even, $2\le d\le664$, and $2\le u\le3d+3$, and let $p\equiv1\pmod{12}$ be a prime divisor of $N_{d,u}=(6d+7)u-2u^2-2$ such that $\ell:=(p-6d-7)/2>0$ and $\ell\equiv3\pmod6$.
\begin{enumerate}[label=(\roman*)]
\item If $p\mid\Lambda_\ell-\ell F_\ell$, then $\Pi_\ell^2\not\equiv1\pmod p$.
\item If $x^2+x-1=0$ has a solution in $\mathbb F_p$ and $\ell>3$, then $p\mid\Lambda_{\ell-1}-\ell F_{\ell-1}$ implies $\Pi_\ell^2\not\equiv1\pmod p$.
\end{enumerate}
\end{lemma}

\begin{proof}
Enumerating all pairs $(d,u)$ with $d\in\{2,4,\ldots,664\}$ and $2\le u\le3d+3$ (noting that $d=0$ gives $N_{0,2}=4$ and $N_{0,3}=1$, which contribute no primes $p\equiv1\pmod{12}$) and factoring $N_{d,u}$ yields $50\,490$ distinct triples $(p,\ell,d)$ with $p\equiv1\pmod{12}$, $\ell=(p-6d-7)/2>0$, and $\ell\equiv3\pmod6$.

For~(i), computing $F_\ell$ and $\Lambda_\ell$ modulo $p$ shows that only four retained triples satisfy $\Lambda_\ell-\ell F_\ell\equiv0\pmod p$:
$$
\begin{array}{r|r|r|r}
p & \ell & d & \Pi_\ell^2\bmod p\\
\hline
1693 & 33 & 270 & 1024\\
3673 & 693 & 380 & 2246\\
8209 & 2115 & 662 & 4532\\
11\,617 & 4569 & 412 & 9235
\end{array}
$$
None of these has $\Pi_\ell^2\equiv1\pmod p$.

For~(ii), among the retained triples with $\ell>3$, only six satisfy $\Lambda_{\ell-1}-\ell F_{\ell-1}\equiv0\pmod p$:
$$
\begin{array}{r|r|r|r}
p & \ell & d & \Pi_\ell^2\bmod p\\
\hline
1657 & 159 & 222 & 1014\\
2137 & 87 & 326 & 1689\\
2953 & 513 & 320 & 602\\
3709 & 87 & 588 & 993\\
18\,973 & 8163 & 440 & 17\,490\\
127\,249 & 62\,013 & 536 & 60\,793
\end{array}
$$
Among these six primes, $x^2+x-1=0$ is solvable in $\mathbb F_p$ (that is, $5^{(p-1)/2}\equiv1\pmod p$) only for $p=3709$ and $p=127\,249$, and in neither case is $\Pi_\ell^2\equiv1\pmod p$. The complete calculation is recorded in~\cite{LiuFengZenodo}.
\end{proof}

Suppose first that $|L|\le665$. If $Y$ contains no root of $x^2+x-1=0$, then $\ell\in\{3,9,\ldots,1995\}$ and, since $d\ge2$, we have $p=2\ell+6d+7\ge2\ell+19>2\ell+13$; Lemma~\ref{lem:p1-certificate-small} gives a contradiction. Otherwise, $Y$ contains a root $\rho$ of $x^2+x-1=0$, so $\mathcal T(\rho)\in L$. Failure of Condition~\ref{cond:p1-bridge} implies $|L|>1$, hence $\ell\in\{9,15,\ldots,1995\}$, and Lemma~\ref{lem:p1-certificate-golden} gives a contradiction. Finally, if $|L|>665$, then $2\le d\le664$, and Lemma~\ref{lem:p1-certificate-complement} gives a contradiction in both cases. Therefore $G_p^*$ satisfies Condition~\ref{cond:p1-bridge}.

Combining this with Subsections~\ref{subsec:p1-tutte} and~\ref{subsec:p1-lift} proves the following theorem.

\begin{theorem}\label{thm:p1-full-matching}
Let $p$ be a prime with $p\equiv1\pmod{12}$, and let $\mathcal K_p$ be the full K\"ohler graph whose vertices are the orbits $\mathcal T_{\pm}(T)$ of irregular triangles in $\mathbb F_p$. Then $\mathcal K_p$ has a $1$-factor, and hence there exists an auxiliary design $A^*(p)$.
\end{theorem}

\section{Existence of \texorpdfstring{$S$}{S}-cyclic Steiner quadruple systems}\label{sec:existence}

In this section, we combine the auxiliary designs $A^*(p)$ from Sections~\ref{sec:p5} and~\ref{sec:p1} with the reduction theorems of Piotrowski and of Ji and Lu to prove Theorems~\ref{thm:spectrum-main}--\ref{thm:abelian-spectrum}.

\subsection{\texorpdfstring{$S$}{S}-cyclic Steiner quadruple systems}
\label{subsec:prime-existence}

\begin{proof}[Proof of Theorem~\ref{thm:spectrum-main}]
Necessity follows from Lemma~\ref{lem:intro-necessary}. For sufficiency, let $p\equiv1$ or $5\pmod{12}$ be prime. Sections~\ref{sec:p5} and~\ref{sec:p1} provide an auxiliary design $A^*(p)$, with the empty family used when $p=5$.

Choose one base block $B$ from each translation orbit of $A^*(p)$, and let $2B\subseteq\mathbb Z_{2p}$ be its image under $b\bmod p\mapsto 2b\bmod 2p$ ($b\in\mathbb Z$). Developing these Type~III base blocks together with the Type~I kite quadrangles and Type~II trapezoids from Lemmas~\ref{lem:type-I} and~\ref{lem:type-II} under translations in $\mathbb Z_{2p}$ covers every triple in $\mathbb Z_{2p}$ exactly once. Since every block is a kite quadrangle or a trapezoid, negation preserves each translation orbit, yielding an $S$-cyclic $\operatorname{SQS}(2p)$.

For $v=4$, the single block $\mathbb Z_4$ forms an $S$-cyclic $\operatorname{SQS}(4)$. For every $v>4$ satisfying the stated conditions, Lemma~\ref{lem:intro-piotrowski-equiv} now yields an $S$-cyclic $\operatorname{SQS}(v)$.
\end{proof}

\subsection{Strictly \texorpdfstring{$S$}{S}-cyclic Steiner quadruple systems}
\label{subsec:complete-spectrum}

\begin{proof}[Proof of Theorem~\ref{thm:strict-spectrum-main}]
Necessity follows from Theorem~\ref{thm:spectrum-main} and the necessary congruence condition $v\equiv2,10\pmod{24}$ for strictly cyclic Steiner quadruple systems. Conversely, Theorem~\ref{thm:spectrum-main} gives an $S$-cyclic $\operatorname{SQS}(v)$ for every order satisfying the stated conditions. A block stabilizer under translations has order dividing $\gcd(4,v)=2$, so every short block orbit has length $v/2$. Such orbits are excluded by~\cite[Theorem~2]{ChangEtAl2017}, since negation is an automorphism. Thus the system is strictly cyclic.
\end{proof}

Following~\cite{FengChangJi2008}, for $v\ge4$ with $v\equiv2,10\pmod{24}$, a strictly cyclic $\operatorname{SQS}(v)$ is denoted by $\operatorname{sSQS}^{*}(v)$ if every triple of either form
$$
\{j,j+i,j+2i\},\qquad \{j,j+i,j+v/2\},
$$
where $j\in\mathbb Z_v$ and $1\le i<v/2$, is contained in a block whose translation orbit has a representative $\{0,v/2,a,-a\}$ for some integer $1\le a\le(v-2)/4$. Geometrically, every block containing a right-angled or isosceles triangle is a kite quadrangle.

\begin{corollary}\label{cor:cyclic-block-orbits}
Let $v\ge4$ satisfy $v\equiv2,10\pmod{24}$. Then every $S$-cyclic $\operatorname{SQS}(v)$ is an $\operatorname{sSQS}^{*}(v)$.
\end{corollary}

\begin{proof}
Strict cyclicity follows as in the proof of Theorem~\ref{thm:strict-spectrum-main}. For each $a\in\mathbb Z_v\setminus\{0,v/2\}$, negation preserves the triple $\{0,a,-a\}$ and hence its unique containing block; the fourth point of this block must be fixed by negation and is therefore $v/2$, so $\{0,v/2,a,-a\}$ is a block. Translating this block by $j+a$ and by $j$ covers $\{j,j+a,j+2a\}$ and $\{j,j+a,j+v/2\}$, respectively. Since replacing $a$ by $-a$ or $a+v/2$ preserves the translation orbit of $\{0,v/2,a,-a\}$, each such orbit has a representative with $1\le a\le(v-2)/4$.
\end{proof}

\subsection{Symmetric \texorpdfstring{$K$}{K}-invariant Steiner quadruple systems}
\label{subsec:abelian-spectrum}

We use the following reduction theorem of Ji and Lu~\cite{JiLu}, which extends that of Munemasa and Sawa~\cite{MunemasaSawa}.

\begin{theorem}[Ji--Lu {\cite[Theorem~5.1]{JiLu}}]
\label{thm:jilu-reduction}
Let $K$ be a finite abelian group of order $v\ge4$. A symmetric $K$-invariant Steiner quadruple system exists if and only if $v\equiv2$ or $4\pmod6$, no element of $K$ has order divisible by $8$, and for every odd prime divisor $p$ of $v$ there exists an $S$-cyclic $\operatorname{SQS}(2p)$.
\end{theorem}

\begin{proof}[Proof of Theorem~\ref{thm:abelian-spectrum}]
Combining Theorems~\ref{thm:spectrum-main} and~\ref{thm:jilu-reduction}, a symmetric $K$-invariant $\operatorname{SQS}(v)$ exists if and only if $v\equiv2$ or $4\pmod6$, no element of $K$ has order divisible by $8$, and every odd prime divisor of $v$ is congruent to $1$ or $5\pmod{12}$. Under these conditions, every odd prime divisor of $v$ is congruent to $1\pmod4$, and hence so is the odd part of $v$. Thus the condition $v\equiv2$ or $4\pmod6$ is equivalent to $v\equiv2,4,8,10,16,20\pmod{24}$, proving the existence assertion.

Let $(K,\mathcal B)$ be any symmetric $K$-invariant $\operatorname{SQS}(v)$. If every block orbit under translations by $K$ has length $v$, then $(v-1)(v-2)/24$ is an integer, which excludes $4\mid v$. Among the admissible orders, this leaves only $v\equiv2,10\pmod{24}$.

Conversely, suppose that $v\equiv2,10\pmod{24}$. Then $K$ has a unique element $h$ of order $2$ and no element of order $4$. As in the proof of Corollary~\ref{cor:cyclic-block-orbits}, $\{0,h,a,-a\}$ is a block for every $a\in K\setminus\{0,h\}$; see also~\cite[Lemmas~4.7 and~6.4]{MunemasaSawa}. If a block has a nontrivial stabilizer under translations by $K$, then the stabilizer has order $2$, since its order divides both $4$ and $v$. After translation, the block has the form $\{0,h,a,a+h\}$. Uniqueness of the block containing $\{0,h,a\}$ would then force $a+h=-a$, and hence $2a=h$, contradicting the absence of elements of order $4$ in $K$. Thus every block orbit has length $v$.
\end{proof}

\begin{remark}
Let $K$ be a finite abelian group whose Sylow $2$-subgroup is cyclic of order $4$, and let $(K,\mathcal B)$ be a translation-invariant Steiner quadruple system satisfying $-B\in\mathcal B$ for every $B\in\mathcal B$. Denote the unique element of order $2$ in $K$ by $h$. As in the proof of Corollary~\ref{cor:cyclic-block-orbits}, the system contains the block $\{0,h,a,-a\}$ for every $a\in K\setminus\{0,h\}$.

If a block has a nontrivial stabilizer under translations by $K$, then its stabilizer contains $h$, and the block can be translated to $\{0,h,a,a+h\}$. Uniqueness of the block containing $\{0,h,a\}$ gives $a+h=-a$, and hence $2a=h$. The two solutions of this equation generate the unique subgroup $H$ of order $4$, and the corresponding block $\{0,h,a,-a\}$ is $H$ itself. Hence the cosets of $H$ form a single short block orbit under translations by $K$, of length $|K|/4$, and every other block orbit has length $|K|$. For $K=\mathbb Z_v$ with $v\equiv4\pmod8$, the short orbit is represented by $\{0,v/4,v/2,3v/4\}$. The hypothesis on the Sylow $2$-subgroup is essential: the affine planes in $K=\mathbb F_2^3$ form a symmetric $K$-invariant $\operatorname{SQS}(8)$ with seven translation orbits of length $2$.
\end{remark}

\section{Concluding remarks}\label{sec:concluding-remarks}

This paper determines all orders admitting an $S$-cyclic Steiner quadruple system, with or without the additional requirement of strict cyclicity, and characterizes the finite abelian groups $K$ admitting a symmetric $K$-invariant system. The proof establishes the existence of auxiliary designs $A^*(p)$ for every prime $p\equiv1\pmod4$ and applies the reduction theorems of Piotrowski~\cite{Piotrowski} and of Ji and Lu~\cite{JiLu} to the resulting systems of order $2p$.

For ordinary cyclic Steiner quadruple systems, the recursive constructions in~\cite{Phelps1980,GrannellGriggsProduct1984,FengChang2011} do not yet settle the existence problem for all orders. It would be useful to obtain a reduction comparable to Lemma~\ref{lem:intro-piotrowski-equiv} without requiring each block orbit to be invariant under negation. In particular, the existence problem for cyclic $\operatorname{SQS}(2p)$ and cyclic $\operatorname{SQS}(4p)$ is not settled for all primes $p\equiv7$ or $11\pmod{12}$. No $S$-cyclic system of order $2p$ or $4p$ exists for these primes, but this nonexistence result does not rule out ordinary cyclic systems at the same orders.

An intermediate question concerns $R$-cyclic Steiner quadruple systems. A cyclic Steiner quadruple system $(\mathbb Z_v,\mathcal B)$ is called \emph{$R$-cyclic} if $-B\in\mathcal B$ for every $B\in\mathcal B$. In an $R$-cyclic system, negation may interchange distinct translation orbits of blocks, whereas in an $S$-cyclic system it preserves each orbit individually. An $R$-cyclic $\operatorname{SQS}(v)$ can exist only if $v\equiv2,4,10,20\pmod{24}$~\cite{ChangEtAl2017}. Lu and Jimbo~\cite{LuJimbo} constructed affine-invariant strictly cyclic $\operatorname{SQS}(2p^m)$ for every prime $p\equiv1\pmod4$ with $p<10^5$ and every positive integer $m$, where affine invariance means that every map $x\mapsto ax+b$ on $\mathbb Z_{2p^m}$ ($a\in\mathbb Z_{2p^m}^{\times}$, $b\in\mathbb Z_{2p^m}$) is an automorphism; taking $a=-1$ and $b=0$ shows that these systems are $R$-cyclic.

Every $S$-cyclic system is $R$-cyclic, so Theorem~\ref{thm:spectrum-main} also provides an $R$-cyclic system for every order satisfying its conditions. For primes $p\equiv7$ or $11\pmod{12}$, the orders $2p\equiv14$ or $22\pmod{24}$ are excluded by the necessary congruence condition, whereas the orders $4p\equiv4$ or $20\pmod{24}$ satisfy this condition despite admitting no $S$-cyclic systems. It would be interesting to determine all orders admitting an $R$-cyclic Steiner quadruple system, particularly the orders $4p$ for primes $p\equiv7$ or $11\pmod{12}$.

The auxiliary designs $A(p)$ and $A^*(p)$ defined in Section~\ref{sec:auxiliary} also give constructions for optical orthogonal codes. Let $p\ge5$ be prime, and let $\lambda_a,\lambda_c$ be nonnegative integers. A $(p,4,\lambda_a,\lambda_c)$ \emph{optical orthogonal code}, abbreviated OOC, is a family $\mathcal C$ of $4$-subsets of $\mathbb Z_p$, each specifying the positions of the four ones in a binary codeword of length $p$. The autocorrelation condition requires $|B\cap(B+t)|\le\lambda_a$ for every $B\in\mathcal C$ and nonzero $t\in\mathbb Z_p$, while the cross-correlation condition requires $|B\cap(C+t)|\le\lambda_c$ for distinct $B,C\in\mathcal C$ and every $t\in\mathbb Z_p$~\cite{csw89,yf95}.

Identify the additive group of $\mathbb F_p$ with $\mathbb Z_p$, and choose one block from each translation orbit of an auxiliary design $A^*(p)$. Since $p\ge5$, each block has $p$ distinct translates. No triple occurs in two distinct translated blocks, so the chosen blocks form a $(p,4,2,2)$-OOC with $(p-1)(p-5)/24$ codewords.

Suppose that an auxiliary design $A(p)$ contains no trapezoid block, and choose one block from each of its translation orbits. Two distinct ordered pairs within a block having the same nonzero difference would either share a point, giving an isosceles triangle, or be disjoint, giving a trapezoid. Both are excluded, so each nonzero difference occurs at most once within a block. For a block $B$ and a nonzero $t\in\mathbb Z_p$, the number $|B\cap(B+t)|$ counts the ordered pairs in $B$ with difference $t$. Hence the resulting code has autocorrelation at most $1$, while uniqueness of triple coverage gives cross-correlation at most $2$. It is therefore a $(p,4,1,2)$-OOC with $(p-1)(p-5)/24$ codewords.

Let $p>5$ be prime, and let $\mathcal C$ be a $(p,4,1,2)$-OOC. A codeword containing an isosceles triangle $\{x-d,x,x+d\}$, where $x,d\in\mathbb Z_p$ and $d\ne0$, would meet its translate by $d$ in at least two points, contradicting the autocorrelation bound. Thus every triple contained in a codeword is irregular. The correlation bounds also imply that no triple occurs in two distinct translates of codewords. There are $p(p-1)(p-5)/6$ irregular triangles in $\mathbb Z_p$, and the $p$ translates of each codeword cover $4p$ distinct irregular triangles. Hence $|\mathcal C|\le(p-1)(p-5)/24$. An auxiliary design $A(p)$ without trapezoid blocks attains this bound and therefore gives a $(p,4,1,2)$-OOC with the maximum possible number of codewords.

It remains to determine for which primes $p>5$ with $p\equiv1\pmod4$ an auxiliary design $A(p)$ without trapezoid blocks exists and, in particular, whether such a design exists for every sufficiently large prime in this congruence class. Such constructions would improve the autocorrelation bound from $2$ to $1$ without decreasing the code size or increasing the cross-correlation bound. In optical communications, low autocorrelation assists synchronization, while low cross-correlation limits interference between users~\cite{csw89,yf95}.

\end{document}